\documentclass[a4paper, 12pt]{article}

\usepackage{amsmath,amssymb,amsthm}
\usepackage{graphicx}
\usepackage[english]{babel}
\usepackage[utf8]{inputenc}
\usepackage[T1]{fontenc}
\usepackage{mathtools}
\usepackage{amsfonts}
\usepackage[margin=1cm]{caption}
\usepackage{bbm}
\usepackage[margin=3cm]{geometry}
\usepackage{xcolor}
\usepackage{enumitem}
\usepackage[hidelinks]{hyperref}
\allowdisplaybreaks
\usepackage{enumitem}
\usepackage{dsfont}
\setlist{nosep}
\usepackage{constants}

\newconstantfamily{c}{symbol=c}

\newcommand{\Z}{\mathbb{Z}}

\newcommand{\R}{\mathbb{R}}
\newcommand{\bbC}{\mathbb{C}}

\newcommand{\dd}{\mathrm{d}}

\newcommand{\rmi}{\mathrm{i}}

\newcommand{\Fcal}{\mathcal{F}}
\newcommand{\Ncal}{\mathcal{N}}

\newcommand{\Ecal}{\mathcal{E}}

\newcommand{\calI}{\mathcal{I}}

\newcommand{\sfone}{\mathsf{1}}

\newcommand{\norm}[1]{\|#1\|}

\newcommand{\sep}{\mathrm{sep}}

\newcommand{\refl}{\mathrm{rfl}}

\newcommand{\RV}{\mathrm{RV}}
\newcommand{\IV}{\mathrm{IV}}

\newcommand{\comb}{\mathsf{Comb}}

\newcommand{\TrigoPolySet}{\mathrm{TPoly}}
\newcommand{\ChargeDensitySet}{\mathrm{CDens}}
\newcommand{\DensCenter}{\mathrm{ctr}}
\newcommand{\DensSquareEnv}{\mathrm{D}}
\newcommand{\SquareSet}{\mathrm{Sq}}
\newcommand{\SquareCover}{\mathcal{S}}
\newcommand{\scale}{\mathrm{sc}}

\DeclareMathOperator{\dist}{dist}

\DeclareMathOperator{\supp}{supp}

\DeclareMathOperator{\diam}{diam}

\DeclareMathOperator{\sinc}{sinc}

\makeatletter
\newcommand{\vast}{\bBigg@{4}}
\newcommand{\Vast}{\bBigg@{5}}
\makeatother

\numberwithin{equation}{section}
\numberwithin{figure}{section}

\renewcommand{\notin}{\not\in}

\newtheorem{theorem}{Theorem}[section]
\newtheorem{corollary}[theorem]{Corollary}

\newtheorem{lemma}[theorem]{Lemma}

\theoremstyle{remark}
\newtheorem{remark}[theorem]{Remark}
\theoremstyle{definition}
\newtheorem{claim}{Claim}

\newcommand{\restrict}{\!\!\restriction\!}

\begin{document}

\title{Quantitative delocalization for solid-on-solid models at high temperature and arbitrary tilt}
\author{Sébastien Ott\thanks{Institute of Mathematics, EPFL, 1015 Lausanne, Switzerland, \texttt{ott.sebast@gmail.com}} \and Florian Schweiger\thanks{Section de mathématiques, Université de Genève, 7-9 rue du Conseil Général, 1205 Genève, Switzerland,  \texttt{florian.schweiger@unige.ch}}}
\date{September 4, 2025}
\maketitle

\begin{abstract}
We study a family of integer-valued random interface models on the two-dimensional square lattice that include the solid-on-solid model and more generally $p$-SOS models for $0<p\le2$, and prove that at sufficiently high temperature the interface is delocalized logarithmically uniformly in the boundary data. Fröhlich and Spencer had studied the analogous problem with free boundary data, and our proof is based on their multiscale argument, with various technical improvements.
\end{abstract}

\section{Introduction and results}

\subsection{Effective interface models and roughening transitions}
Consider a model of lattice statistical mechanics for which there are at least two Gibbs states (phases). When one forces these phases to coexist, for example by imposing suitable boundary conditions, a natural object of interest is the interface between them. If the original model is defined on a $d+1$-dimensional lattice, the resulting interface will be a $d$-dimensional object. 

Typically, the behavior of the interface is most intricate in $d=2$. Namely a possible candidate for a scaling limit of the interface is the (continuum) Gaussian free field, and this object has critical dimension $d=2$, in the sense that it is delocalized there, but only at logarithmic rates. So it is natural to wonder if the underlying discrete interface is delocalized as well, or if it stays localized. As higher temperatures lead to larger fluctuations, the answer to this question could depend on the temperature of the model, and one says that the model has a roughening transition if there is a non-trival transition between localized and delocalized behavior of the interface as the temperature increases.

Unfortunately, for models of interests in statistical mechanics, while it is often feasible to establish the existence of a localized regime at very low temperature, it is usually very hard to prove existence of a delocalized regime (and hence also a roughening transition). For example, when one considers the Ising model on a $2+1$-dimensional box with plus boundary conditions on the top half and minus boundary conditions on the bottom half (the so-called Dobrushin boundary conditions), then at any temperature below the critical temperature $T_c$ there will be an interface between the plus and minus phases. While it was predicted that at zero tilt there is a roughening transition at $T_r\approx 0.57T_c$ \cite{WGL75}, the only rigorous results concern the existence of a localized regime \cite{D72,vB75}, and it is a longstanding open problem to verify the existence of a roughening transition.

Because of this, it is common in the literature to make a simplifying assumption, namely to assume from the beginning that the interface is given by the graph of a function $\phi:\Z^d\to\Z$. This leads to so-called (integer-valued) effective interface models. For instance, if one conditions the Dobrushin interface to be given by the graph of a function, one obtains the so-called solid-on-solid (SOS) model.

For such effective interface models, one can again study the existence of roughening transitions. Once again, it is known in wide generality that zero-tilt interfaces are localized at low temperature \cite{BW1982}. The converse question of (logarithmic) delocalization for $d=2$ at high temperature is harder, but was answered positively for several important models including the Gaussian free field (GFF) and SOS model in \cite{FS81a}, and there has been recent progress for some other models such as uniform Lipschitz functions and the height function associated to six-vertex models \cite{CPST21,GM21,DKMO24,GL25}.

\subsection{\texorpdfstring{\(p\)-SOS models}{p-SOS models}}
\label{subsec:p_SOS}
The main focus of this work is on the so-called $p$-SOS interface models. These are a family of effective, real- or integer-valued gradient interface models with interaction term $|\cdot|^p$. Notable special cases are the GFF for $p=2$ and the usual SOS model for $p=1$. 

Given a finite domain $\Lambda\subset\Z^d$ and boundary values $\xi\in\R^{\Z^d}$, one can consider the Hamiltonian
\[H_\xi(\phi)=\sum_{\{i,j\} \subset \Lambda, i\sim j} |\phi_i-\phi_j|^p +\sum_{i\in \Lambda,j\in \Lambda^c, i\sim j }|\phi_i-\xi_j|^p\]
defined on functions $\phi\in\R^\Lambda$.

Then the real-valued $p$-SOS model at inverse temperature $\beta>0$ is given by
\[d \mu_{\Lambda;\beta}^{p-\mathrm{SOS},\xi,\RV}(\phi) \propto e^{-\beta H_\xi(\phi)} d\phi.\]
Its integer-valued version is (formally) equal to the real-valued model conditioned to take integer values in $\Lambda$. Rigorously, it can be defined via
\[d \mu_{\Lambda;\beta}^{p-\mathrm{SOS},\xi,\IV}(\phi) \propto e^{-\beta H_\xi(\phi)} \prod_{i\in \Lambda}\sum_{n\in \Z}d\delta_n(\phi_i).\]
Even more generally one can require not that $\phi_i\in\Z$ but that $\phi_i-\zeta_i\in\Z$ where $\zeta_i\in\R^\Lambda$ is some given field, i.e.~that the constraint of being integer-valued is shifted along the fibers. The corresponding measure is then given by
\[d \mu_{\Lambda;\beta}^{p-\mathrm{SOS},\xi,\zeta,\IV}(\phi) \propto e^{-\beta H_\xi(\phi)} \prod_{i\in \Lambda}\sum_{n\in \Z}d\delta_n(\phi_i-\zeta_i).\]

Our main interest in these models is in the critical dimension $d=2$. In that case, the real-valued model is delocalized at any temperature, but only barely so in the sense that the variance on a finite box of sidelength $N$ diverges logarithmically in $N$. See for example~\cite{DS1980,FP1981} for classical results with some sub-optimal speed and heavy restrictions on the pair interaction potential, and~\cite{ISV2002, MP2015} for more recent results with optimal speed and relaxed hypotheses on the potential. As mentioned in the previous subsection, one can wonder if this delocalization survives at high temperature, and if there is a roughening transition.

It is believed that the answer to this question depends on the tilt of the model. Let us begin by considering the zero-tilt case, corresponding to the choice $\xi=0$. In that case it follows from a standard low temperature expansion that for large $\beta$ the model is localized in the sense that the variance of the field is bounded uniformly in $N$ (as shown for much more general height functions in \cite{BW1982}). For small $\beta$, on the other hand, it is believed that the model is delocalized with logarithmically diverging variances. This was shown rigorously for the case of the GFF and the SOS model (i.e.~for $p=2$ and $p=1$) in a very influential work by Fröhlich and Spencer \cite{FS81a}\footnote{The proof for the SOS case in \cite{FS81a} has some issues, though, as the proof in \cite[Section 7.2]{FS81a} does not actually give the estimate in \cite[Theorem 7.3]{FS81a} (cf. Remark~\ref{r:techimprovement} below for details). Because of this, the results in \cite{FS81a} do not directly imply logarithmic delocalization of variances for the SOS model. As a by-product of our main results, we fill this gap.}.

In recent years, there have been exciting further developments. Namely using an elegant percolation argument from \cite{L22}, in \cite{AHPS21} and independently in \cite{vEL23,vEL23b} alternative proofs of delocalization for various interfaces (including the GFF) at high temperature and zero tilt were given. These results are qualitative, but together with a dichotomy result from \cite{L23} they imply sharpness of the roughening transition and logarithmic delocalization throughout the delocalized phase. Moreover, in the case of the GFF at very high temperature much more than just delocalization is known: In \cite{BPR22a,BPR22b} renormalization-group methods were used to prove that the scaling limit of the (infinite volume, zero-tilt) interface is the continuum GFF.

The case of non-trivial tilt $u$, given by $\xi_i=u\cdot i$ for $i\in\Z^d$, where $u\in\R^d$ with $u\neq0$, is quite different. Namely it is expected that the field now is delocalized even if $\beta$ is large. The reason for this is that the tilted boundary conditions destroy the uniqueness of the ground state (i.e.~of the minimizers of $H_\xi$), and the resulting fluctuations of the ground state already manage to delocalize the field. However, for large $\beta$ the only relevant rigorous result is the recent \cite{LL24}, where for certain perturbations of the SOS-model (but not the SOS-model itself) convergence in a scaling limit to the continuum GFF (and hence in particular logarithmic delocalization) is shown. For small $\beta$, on the other hand, in the case of the SOS model, Lammers and the first author \cite{LO24} have shown qualitative delocalization (in the sense that there is no infinite-volume Gibbs state). This result does not apply in finite volume, though, and does not give any quantitative lower bound on the variances\footnote{The dichotomy statement of \cite{L23} does not apply in the case of non-zero tilt because it is based on relating the height function to certain correlated percolation models. For that purpose it is essential that the law of the interface is invariant under $\phi\mapsto-\phi$.}. Logarithmic delocalization in finite volume was previously known only for the case of the GFF, due to work by Garban and Sepúlveda \cite{GS24}\footnote{The relevant result \cite[Theorem 1.8]{GS24} is stated there for free or Dirichlet boundary data, but in the Gaussian case one can reduce general boundary data to Dirichlet boundary data by a straightforward change of variables.}.

Let us also mention that the one-dimensional $p$-SOS model (but possibly even with long-range interactions) has recently been studied in \cite{CDL24}, and they prove upper and lower bounds on the variance. Their proofs rely on rewriting $p$-SOS models as mixtures of GFFs with random conductances, and bounding the fluctuations of the latter. This is managable in $d=1$, but much harder in $d=2$ where one is confronted with roughening transitions for the random-conductance GFFs.

\subsection{Main results}

As explained above, all existing results for delocalization of integer-valued $p$-SOS models at high temperature are restricted to the zero-tilt case, apply only to the GFF ($p=2$) or are qualitative. In particular, even for the key example of the usual SOS model at any non-zero tilt, it was only known that the interface delocalizes qualitatively \cite{LO24}.

Our main result is concerned with all $p$-SOS models for $0<p\le2$ (including in particular the usual SOS model) in $d=2$ and we prove logarithmic delocalization at high temperature, uniformly in the boundary data and the fiber shifts.

\begin{theorem}
    \label{thm:main:p_SOS}
    Suppose that $d=2$. Then there is \(\beta_0>0\) such that for any for any $0<p\le 2$, any \(\beta < \beta_0\), any $\Lambda\subset\Z^2$ finite, any \(\xi,\zeta \in \R^{\Z^2}\) and any \(f\in \R^{\Lambda}\) we have
    \begin{equation}\label{eq:pSOSvar}
        \mu_{\Lambda;\beta}^{p-\mathrm{SOS}; \xi,\zeta} \big( f\cdot \phi\, ;\, f\cdot \phi \big)
        \geq  \frac{1}{\beta_{\rm eff}} f \cdot \Delta_{\Lambda}^{-1} f,
    \end{equation}
    for some \(\beta_{\rm eff}>0\). In fact we can take \(\beta_{\rm eff} \le \beta^{c} \) for some universal constant $c>0$.
    
    Moreover, if \(p\in (1,2]\), $\Lambda=-\Lambda$ and
    \begin{equation*}
        f_i = f_{-i},\quad \xi_{i} = -\xi_{-i},\quad \zeta_i=-\zeta_{-i},
    \end{equation*}one has
    \begin{equation}\label{eq:pSOSexpm}
        \mu_{\Lambda;\beta}^{p-\mathrm{SOS}; \xi,\zeta} \big( e^{f\cdot \phi} \big)
        \geq
        \exp\left(\frac{1}{2\beta_{\rm eff}}f \cdot \Delta_{\Lambda}^{-1} f\right)
    \end{equation}for \(\beta <\beta_0\). Finally, \eqref{eq:pSOSexpm} also holds for \(p=1\) provided $f$ is small enough in the sense that \(\|\nabla(\Delta_\Lambda^{-1}f)\|_\infty<\beta\).
\end{theorem}

Note that by replacing $f$ by $\epsilon f$ and expanding in powers of $\epsilon$ one directly obtains that \eqref{eq:pSOSexpm} implies \eqref{eq:pSOSvar}.

In addition, from \eqref{eq:pSOSvar} one easily obtains logarithmic divergence of spins. For instance, if we consider the field on the domain $\Lambda_N=\{-N,\ldots N\}^2$, then for \(\beta<\beta_0\) and uniformly in the boundary data $\xi$ and the shift in the fibers $\zeta$ one has
\begin{equation}
    \mu_{\Lambda_N;\beta}^{p-\mathrm{SOS}; \xi,\zeta} \big( \phi_0^2 \big)
    -\mu_{\Lambda_N;\beta}^{p-\mathrm{SOS}; \xi,\zeta} \big( \phi_0 \big)^2
    =\mu_{\Lambda_N;\beta}^{p-\mathrm{SOS}; \xi,\zeta}(\phi_0;\phi_0)\geq
    \frac{c}{\beta_{\rm eff}}\ln(N).
\end{equation}

Our results are in fact more general and apply to a larger class of models for which one can find an analytic function interpolating the Hamiltonian with certain properties. We refer to Theorems~\ref{thm:LB_trigo_poly_non_sym} and~\ref{thm:LB_trigo_poly_sym}, and Section~\ref{sec:trigo_poly_to_sum} for the relevant technical statements. 

Let us remark at this point, however, that one important further example that falls into our framework is the height function dual to the XY model (as studied in \cite{vEL23b,AHPS21}, with the latter reference introducing the name ``integer-valued Bessel field''.

\subsection{Main ideas of the proofs}\label{subsec:mainideas}

Our proof method is based on the one for the zero-tilt case in \cite{FS81a} (with additional input from \cite{W19,G23}), and also relies on the argument for the GFF case with arbitrary tilt in \cite{GS24}. Our key new contributions are a novel symmetry argument and a technical improvement to the estimate for the moment-generating function in \cite[Section 7.2]{FS81a}. 

Below we give a quick summary of the proof and of our main ideas.

As a first step, one can replace the constraint of being integer-valued by considering the continuum-valued model with a periodic potential $\lambda$.
If one can prove estimates for the latter models are sufficiently uniform in $\lambda$, one can pass to limits and deduce the result for the integer-valued model. 
The next step is to re-write the site-wise ``periodic potential'' \(\lambda\) as a suitable convex combination of multipole potentials (i.e.~a coarse-graining step). It gives a re-writing of the partition function as a convex combination of partition functions with coarse-grained potentials. The third step is to extract the leading contribution in each of the partition functions, by means of complex translations. These steps are done in the same fashion as in the original argument of \cite{FS81a} in the version of \cite{W19} (the former work was concerned with free boundary conditions, the latter explained in detail the modifications necessary to study Dirichlet boundary conditions), and using an observation from \cite{G23} that allows to study arbitrary observables.

The next step is where our treatment differs:~\cite{W19} considers the \(p=2\) case which induces several simplifications. In the general case, when expanding $\int d\mu_{\Lambda;\beta}^{\xi,\RV}(\phi)\lambda(\phi) \exp(f\cdot\phi)$ into multipoles and doing the complex shifts, one gets additional error terms $f\cdot a_\rho$ (which are not present in the GFF case) that arise from applying the complex shift to $\exp(f\cdot\phi)$. In \cite[Section 6,7]{FS81a} these terms are estimated rather brutally, arguing that by dropping these terms altogether one only loses a global constant factor for reasonable choices of \(f\). While this is sufficient to give a lower bound on the two-point function of the XY model (which is dual to a Height Function model), the loss of a global constant at this stage gives only a result of the form
\begin{equation*}
    \int d\mu_{\Lambda;\beta}^{\xi,\RV}(\phi)\lambda(\phi) \exp(f\cdot\phi)
    \geq
    ce^{\frac{1}{2\beta'}f\cdot \Delta^{-1} f},
\end{equation*} with some \(c>0\) (see~\cite[(7.35), (7.40)]{FS81a}). In particular, this does not give~\cite[Theorem 7.3]{FS81a}, which is stated with \(c=1\). As a result, the lower bound on the variance, which is obtained by a second order Taylor expansion of the bound on the Laplace transform, does not follow as in~\cite[(7.18)]{FS81a}. We do not know how to deduce the variance bound (in either the zero-tilt or non-zero tilt case) without \(c=1\). We therefore proceed more cautiously than~\cite[Section 7]{FS81a}: we treat the terms $f\cdot a_\rho$ similarly to the terms $\sigma\cdot\rho$ in \cite{FS81a}. That is, we Taylor-expand everything to first order, and argue that the first-order terms are antisymmetric while the error term is controlled. See Sections~\ref{subsec:Taylor},~\ref{subsec:LB_without_sym} below for details.

Aside from this ``technical improvement'', we need a new idea to deal with
the first order terms. In~\cite{FS81a}, it is used that they are antisymmetric while the law of the field is invariant under $\phi\mapsto-\phi$, and so Jensen's inequality allows to bound them from below. However, for general boundary values the field no longer has the symmetry $\phi\mapsto-\phi$, and so this argument breaks down.
Our second new idea now is to work with the symmetry $\phi\mapsto-\phi^\refl$, where $\phi^\refl_i=\phi_{-i}$. If the domain and the boundary data happen to be antisymmetric, then can use this symmetry and proceed along the lines of \cite{FS81a} to obtain a lower bound on the Laplace transform (and eventually \eqref{eq:pSOSexpm}).

This alternative symmetry argument, however, requires strong symmetry assumptions on the domain and the boundary values. In order to obtain the much more generally applicable bound in \eqref{eq:pSOSvar}, we use instead an idea from \cite{GS24}. Namely they observed in the context of the GFF, that if after renormalizing the multipole expansion one has a good enough control on the remaining error terms, then one can replace $f$ by $\epsilon f$ already at this point and expand, which together with the Cauchy-Schwarz inequality leads to a lower bound on the variances. In order to apply this idea in our setting, however, it is crucial that one does not lose a constant factor in the renormalization as in \cite{FS81a}. This is where our ``technical improvement'' comes in again: We manage to avoid the loss of a constant factor, and hence are able to close the argument.

\medskip

Below, in Section~\ref{sec:technicalresults} we introduce the general class of height functions that our technique applies to, and state our main results in full generality. In Section~\ref{sec:chargesrenorm} we recall the multiscale method from \cite{FS81a} and state precise versions of various results from \cite{FS81a,W19} that we will need. Our main new contribution is in Section~\ref{sec:improvements}, where we give details on our technical improvements to the multiscale analysis, and use it to give proofs for the general results in Section~\ref{sec:chargesrenorm}. Finally, we explain in Section~\ref{sec:trigo_poly_to_sum} how Theorem~\ref{thm:main:p_SOS} follows from our general result.

\section{Main technical results}\label{sec:technicalresults}

As mentioned, our results apply not just to $p$-SOS models, but to a more general class of height functions. In this section we will introduce the necessary notation and state our technical results precisely.

\subsection{General notations}
We will denote by \(\dist\) the nearest-neighbour graph distance on \(\Z^2\). For \(i,j\in\Z^2\), we denote \(i\sim j\) for \(\dist(i,j) = 1\). For \(\Lambda\subset \Z^2\), denote \(\Lambda^c = \Z^2\setminus \Lambda\). Denote \(\Delta= -\nabla\cdot \nabla\) the usual (positive definite) nearest-neighbour Laplacian on \(\Z^2\), and \(\Delta_{\Lambda}\) its restriction to \(\Lambda\) with Dirichlet boundary conditions (so that \(\Delta_{\Lambda}^{-1}\) is well defined). For \(f,g:\Z^2\to \R\), denote \(f\cdot g = \sum_{i\in \Z^2} f_i g_i\) their scalar product, and extend it to gradients in the obvious manner. Also write
\begin{equation*}
    \norm{f}_{2}^2 = \sum_{i}f_i^2,\qquad \norm{\nabla f}_{2}^2 = \sum_{ i\sim j}(f_i-f_j)^2.
\end{equation*}We will denote by \(f\restrict_V\) the restriction of \(f\) to \(V\), and similarly by \((\nabla f)\restrict_{V}\) the restriction of the gradient of \(f\) to the edges of \(V\):
\begin{gather*}
    f\restrict_V(i) = f_i \text{ if } i\in V,\qquad f\restrict_V(i) = 0 \text{ else},
    \\
    (\nabla f)\restrict_{V}(ij) = f_j-f_i \text{ if } i\sim j, i,j\in V,\qquad (\nabla f)\restrict_{V}(ij) = 0 \text{ else}.
\end{gather*}Let \(M,\alpha\) be fixed with
\begin{equation*}
    M= 2^{16}, \quad \tfrac{3}{2}<\alpha < 2.
\end{equation*}
Regarding constants, ``universal constants'' are allowed to depend on \(M\) as it is a given number. To track dependency, we label the constants. Unlabeled constants (i.e.~without sub-index, e.g.~\(c,c',C,C'\)) have a value which can change from line to line and do not depend on the parameters that are allowed to vary. Their use is restricted to local calculations. For both labeled and unlabeled constants, a capital/small letter typically suggests that the corresponding constant is large/small.

Our main tool is the multipole expansion of Fröhlich and Spencer~\cite{FS81a}, with the recent upgrades of Wirth~\cite{W19}. We present here  the relevant notions/claims, and give either proofs or references to proofs of the statements. As in~\cite{W19}, we use the presentation of~\cite{KP17}. In particular, we carry out most of the proof in some abstract framework, described in the next subsection.

\subsection{Models}
\label{subsec:models}

The goal is to study integer valued random height functions in finite volume \(\Lambda\subset \Z^2\) with arbitrary boundary conditions. For \(\beta >0\), let
\begin{equation*}
    I_{\beta}:\Z\to \R
\end{equation*} be a positive, even function with at least stretched exponential tails. We moreover assume that there is an even function \(g:\R\to \R_+\) with
\begin{equation}\label{eq:g_quadr}
        g(at)\le \Cl{g_quadr}t^2g(a)\quad\forall a\ge0, t\in(0,1)
\end{equation}
for some \(\Cr{g_quadr}\geq 1\), as well as \(\epsilon_{\beta}>0\)
\(c_{\beta},c_{\beta}'> 0\)
such that \(I_{\beta}\) can be extended to an analytic function on the complex strip \(\{z\in \bbC:\ |\Im(z)| <\epsilon_{\beta}\}\), and that this extension (which we also denote \(I_{\beta}\)) satisfies the following points. 
\begin{enumerate}
    \item \label{item:Ibeta:reals}\(I_{\beta}\) is even, and real positive on \(\R\). Moreover, \(I_{\beta}\) is a Schwartz function.
    \item \label{item:Ibeta:imaginary_growth} For any \(a,x\in \R\) with \(|a|< \epsilon_{\beta} \),
    \begin{equation}
        \Big|\frac{I_{\beta}(x+\rmi a)}{I_{\beta}(x)}\Big| \leq e^{c_{\beta} g(a)}.
    \end{equation}
    \item \label{item:Ibeta:derivatives} For \(m=1,2\), and \(|a|< \epsilon_{\beta}\),
    \begin{equation}
        \Big|\frac{\dd^m}{\dd x^m} \ln\frac{I_{\beta}(x+\rmi a)}{I_{\beta}(x)} \Big|\leq \Cl{lnIratio_deriv}c_{\beta} (1+g(a)).
    \end{equation}
    for some universal constant $\Cr{lnIratio_deriv}>0$.
    \item \label{item:Ibeta:sub_quad} For any \(x\in \R\),
    \begin{equation}
        \Big|\frac{\dd^2}{\dd x^2} \ln I_{\beta}(x) \Big|\leq c_{\beta}'.
    \end{equation}
\end{enumerate}

In the above setting we will often assume that we are given an absolute constant $\Cl[c]{lowerboundmainthm}>0$ and a constant $\Cl[c]{quantsmallgamma}>0$ (whose value might depend on $\Cr{g_quadr},\Cr{lnIratio_deriv}$ only) and \(\gamma_{\beta}\in(0,\epsilon_{\beta})\) such that
\begin{equation}\label{eq:def:gamma_beta}
    \max\left(\frac{1}{\gamma_\beta},\frac{c_\beta g(\gamma_{\beta})}{\gamma_{\beta}},
    c_\beta'e^{\Cr{lowerboundmainthm}\gamma_{\beta}}\right)\le \Cr{quantsmallgamma},
    \quad \Cr{lnIratio_deriv}c_{\beta}\leq 1.
\end{equation}
We will indicate explicitly whenever that assumption is required. The condition \(\Cr{lnIratio_deriv}c_{\beta}\leq 1\) can be replaced by \(\Cr{lnIratio_deriv}c_{\beta}\) less than any fixed constant (at the expense of changing the value of $\Cr{quantsmallgamma}>0$ for which our results hold).

\begin{remark}\label{r:choiceparam}
As the above assumptions might seem very complicated at first glance, let us help the reader's intuition by explaining right away which choices of parameters one can make in our key example, the $p$-SOS model (for $0<p\le2$). In that case we must have \(I_{\beta}(n) = e^{-\beta |n|^p}\) for \(n\in\Z\).

We will see later (in Lemma~\ref{p:heightfunctextension}) that for $\beta\le1$, say, one can find an analytic extension $I_\beta$ that allows to take $g(a)=a^2(1+e^{2\pi|a|})$ (which satisfies \eqref{eq:g_quadr} with $\Cr{g_quadr}=1$) as well as $c_\beta=c_\beta'=C\beta^{1/3}$, $\epsilon_\beta=\frac{1}{2\beta^{1/3}}$ (with the constants independent of $p$). This allows to take, for example $\gamma_\beta=c|\log\beta|$ for some sufficiently small universal $c$, and then \eqref{eq:def:gamma_beta} holds as soon as $\beta$ is small enough. This ultimately leads to the estimate $\beta_{\rm eff}\le \exp\left(-c'|\log\beta|\right)=\beta^{c'}$ for another universal constant $c'$. All this is consistent with Theorem~\ref{thm:main:p_SOS}.
\end{remark}
\begin{remark}\label{r:difftoFS}
Our assumptions are based on those in \cite{FS81a}, however there are some differences:
\begin{itemize}
\item
    The analogue of condition~\ref{item:Ibeta:derivatives} in~\cite[Section 7.2]{FS81a} has a factor $g(a)$ in place of $(1+g(a))$, i.e.~is stronger. However, unlike what is claimed in~\cite[Appendix C]{FS81a}, the extension $I_\beta$ for the SOS-model does not actually satisfy this assumption.
    Our weaker assumption (which is similar to the one used in~\cite[Proof of (6.1)]{FS81a} for the XY-model), bypasses this problem.
\item
    Condition~\ref{item:Ibeta:sub_quad} is not explicitly mentioned in~\cite{FS81a}, but it is crucially used in the last step of their argument, cf. \cite[(6.28)]{FS81a}. This condition is not a very surprising one: it is the same as the one as the one encountered in the ``early rigorous versions'' of Mermin-Wagner argument in the context of continuous height function models, see for example~\cite{FP1981}.
    \item
    Moreover, \cite{FS81a} considers only two specific examples of functions $g$, namely $g(a)=a^2$ and $g(a)=\begin{cases}Ca^2& |a|\le1\\ Ce^{2\pi|a|}&\text{else}\end{cases}$. Our assumption \eqref{eq:g_quadr} is new and captures the only property of $g$ that is needed for the argument.
\end{itemize}
\end{remark}

For \(\Lambda\subset \Z^2\) finite and \(\xi\in \R^{\Z^2}\), define the weights \(\calI_{\Lambda;\beta}^{\xi}:\R^{\Lambda}\to \R_+\) 
\begin{equation*}
    \calI_{\Lambda;\beta}^{\xi}(\phi) = \prod_{\{i,j\} \subset\Lambda, i\sim j} I_{\beta }(\phi_i-\phi_j) \prod_{i\in \Lambda,j\in \Lambda^c, i\sim j }I_{\beta }(\phi_i-\xi_j),
\end{equation*}
as well as the probability measures \(\mu_{\Lambda;\beta}^{\xi,\RV}\), \(\mu_{\Lambda;\beta}^{\xi,\IV}\), and \(\mu_{\Lambda;\beta}^{\xi,\zeta,\IV}\) (with \(\zeta\in \R^{\Z^2}\)) supported on, respectively, \(\R^{\Lambda}\), \(\Z^{\Lambda}\), and \(\Z^{\Lambda}\) via
\begin{gather*}
    d \mu_{\Lambda;\beta}^{\xi,\RV}(\phi) \propto \calI_{\Lambda;\beta}^{\xi}(\phi) \prod_{i\in \Lambda}d\phi_i,
    \qquad
    d \mu_{\Lambda;\beta}^{\xi,\IV}(\phi) \propto \calI_{\Lambda;\beta}^{\xi}(\phi) \prod_{i\in \Lambda}\sum_{n\in \Z}d\delta_n(\phi_i),
    \\
    d \mu_{\Lambda;\beta}^{\xi,\zeta,\IV}(\phi) \propto \calI_{\Lambda;\beta}^{\xi}(\phi) \prod_{i\in \Lambda}\sum_{n\in \Z}d\delta_n(\phi_i-\zeta_i).
\end{gather*}

When \(\RV/\IV\) are omitted, it is set to be \(\RV\). We will sometimes leave \(\beta\) implicit in the measures, so we often remove it from the notation.

\subsection{Main result}
\label{subsec:main_result_trigo_poly}

The constraint that $\phi_i-\zeta_i\in\Z$ is very degenerate. Since, formally,
\[\sum_{n\in\Z}\delta_n(x)=\sum_{q\in\Z}e^{2\pi i q x}=1+2\sum_{q=1}^\infty \cos(2\pi q x)\]
it is natural to replace this constraint by a periodic density given by a finite trigonometric polynomial. Once we have estimates in that case that are sufficiently uniform in the trigonometric polynomial, it will be possible to pass to the limit and deduce Theorem \ref{thm:main:p_SOS}. Thus our main technical results will be stated for continuum-valued measures with reference measure $d \mu_{\Lambda;\beta}^{\xi,\RV}(\phi)$, and with a finite trigonometric polynomial as the density.

In detail, say that \(\lambda : \bbC \to \bbC\) is a real, even, normalized trigonometric polynomial if there are an integer \(N\geq 1\), and numbers \(\hat{\lambda}(q)\in \R\), \(q= 1,\dots, N\) such that
\begin{equation*}
    \lambda(x) = 1 + 2\sum_{q=1}^N \hat{\lambda}(q) \cos(q x).
\end{equation*}Denote by \(\TrigoPolySet\) the set of real, even, normalized trigonometric polynomials. If we are given a set \(\Lambda\) and trigonometric polynomials \(\lambda_i\in \TrigoPolySet\), \(i\in \Lambda\), set
\begin{equation*}
    \lambda_{\Lambda} : \bbC^{\Lambda} \to \bbC,
    \quad
    \lambda_{\Lambda}(x) = \prod_{i\in \Lambda} \lambda_{i}(x_i).
\end{equation*}

For \(\Lambda,\xi,\zeta\) as in Section~\ref{subsec:models}, and \(\lambda_{\Lambda}\) a trigonometric polynomial, define (when possible)
\begin{equation}
\label{eq:def:trigo_modif_meas}
    d\mu_{\Lambda;\beta,\lambda_{\Lambda}}^{\xi,\zeta}(\phi)
    :=
    \frac{d\mu_{\Lambda;\beta}^{\xi}(\phi)\lambda_{\Lambda}(\phi-\zeta)}{\int d\mu_{\Lambda;\beta}^{\xi}(\varphi)\lambda_{\Lambda}(\varphi-\zeta)}.
\end{equation}Note that, while being normalised, these are in general signed measures (and might a priori not even be well-defined). 
We can now state the main technical result of this work.
\begin{theorem}
    \label{thm:LB_trigo_poly_non_sym}
In the setting described in Section~\ref{subsec:models}, there are constants $\Cr{lowerboundmainthm}>0,\Cr{quantsmallgamma}>0$ (depending on $\Cr{g_quadr},\Cr{lnIratio_deriv}$ only) such that whenever there is $\gamma_\beta$ such that \eqref{eq:def:gamma_beta} holds, we have the following.

Let \(\lambda_i\in \TrigoPolySet, i\in \Lambda\) be such that \(\hat{\lambda}_i(q)\leq 1\) for all \(i,q\). Then, all of the following hold.
    \begin{enumerate}
        \item The expression~\eqref{eq:def:trigo_modif_meas} is well defined.
        \item For any \(\Lambda\subset \Z^2\) finite and simply connected, any \(\xi,\zeta\in \R^{\Z^2}\), and any \(f\in \R^{\Lambda}\),
        \begin{equation}\label{eq:LB_trigo_poly_var}
            \int d\mu_{\Lambda;\beta,\lambda_{\Lambda}}^{\xi,\zeta}(\phi) (f\cdot \phi)^2 - \Big(\int d\mu_{\Lambda;\beta,\lambda_{\Lambda}}^{\xi,\zeta}(\phi) f\cdot \phi\Big)^2
            \geq e^{\Cr{lowerboundmainthm}\gamma_{\beta}} f\cdot \Delta_{\Lambda}^{-1} f,
        \end{equation}where \(\gamma_{\beta}\) is as in Section~\ref{subsec:models}.
    \end{enumerate}
\end{theorem}
The second point is \emph{almost} a variance bound: the only thing preventing it from being one is that we are not a priori looking at a probability measure.

Under symmetry assumptions we can strengthen this ``variance bound'' to a bound on ``exponential moments'' (with the same caveat that we might not be looking at a probability measure). To state the result, we define
\begin{equation}\label{eq:def_reflectionsymmetry}
    \Lambda^{\refl} = \{-x:\ x\in \Lambda\}, \quad f_i^{\refl} = f_{-i}.
\end{equation}
Then we have the following bound.
\begin{theorem}    \label{thm:LB_trigo_poly_sym}
In the setting of Theorem~\ref{thm:LB_trigo_poly_sym} assume additionally that
\begin{equation*}
    f^{\refl} = f,\quad \Lambda = \Lambda^{\refl},\quad \xi = -\xi^{\refl}, \quad \zeta=-\zeta^{\refl},\quad \lambda_{i} = \lambda_{-i}.
\end{equation*}
Then we also have that
 \begin{equation}\label{eq:LB_trigo_poly_expmom}
    \int d\mu_{\Lambda;\beta,\lambda_{\Lambda}}^{\xi,\zeta}(\phi) e^{f\cdot \phi}
    \geq \exp\Big(\tfrac12e^{\Cr{lowerboundmainthm}\gamma_{\beta}} f\cdot \Delta_{\Lambda}^{-1} f\Big)
\end{equation}
whenever the left-hand side is well defined.

\end{theorem}
Note that \eqref{eq:LB_trigo_poly_var} follows directly from \eqref{eq:LB_trigo_poly_expmom} by replacing $f$ with $\epsilon f$ and expanding in powers of $\varepsilon$. Thus the conclusion of Theorem~\ref{thm:LB_trigo_poly_sym} is strictly stronger than the one of Theorem~\ref{thm:LB_trigo_poly_non_sym}.

The proof of Theorems~\ref{thm:LB_trigo_poly_non_sym} and~\ref{thm:LB_trigo_poly_sym} occupies Sections~\ref{sec:chargesrenorm} and \ref{sec:improvements}.

\section{Charges and renormalization}\label{sec:chargesrenorm}

Our proofs are an adaption of the multiscale method of \cite{FS81a}, with some modifications from \cite{W19}. In this section we state all relevant definitions and recall various results from \cite{FS81a,W19} and the alternative presentation in \cite{KP17}. None of the results in this section are novel.

\subsection{Charge densities and square covering}

The multiscale method of \cite{FS81a} proceeds by a complicated coarse-graining argument that is based on combinatorics of (integer-valued) charge densities. We begin by collecting the relevant definitions in this subsection.
\subsubsection*{Distances, diameter, and related notations}

For \(f,g: \Z^2\to \bbC\), define
\begin{itemize}
    \item the \emph{support} of \(f\): \(\supp f = \{i\in \Z^2:\ f(i) \neq 0\}\);
    \item the \emph{support of the gradient} of \(f\):
    \begin{equation*}
        \supp \nabla f = \{i\in \Z^2: \ \exists j\sim i\ f(i)\neq f(j)\};
    \end{equation*}
    \item the \emph{diameter} of \(f\): \(\diam f = \sup_{i,j\in \supp f} \dist(i,j)\);
    \item the \emph{distance} between (the supports of) \(f\) and \(g\):
    \begin{equation*}
        \dist(f,g) = \inf_{i\in \supp f, j\in \supp g} \dist(i,j).
    \end{equation*}
    \item for \(\Lambda\subset \Z^2\), the \(\Lambda\)-\emph{modified diameter} of \(f\):
    \begin{equation*}
        \dd_{\Lambda}(f) = \begin{cases}
        \diam(f) & \text{ if } f \cdot 1 = 0 \text{ (} f \text{ neutral)},
        \\
        \max\big(\diam(f),\dist(\supp f, \Lambda^c)\big) & \text{ else}.
    \end{cases}
    \end{equation*}
\end{itemize}

\subsubsection*{Charges densities and charge ensembles}

\begin{itemize}
    \item A \emph{charge density} is a non-zero function \(\rho : \Z^2 \to \Z\) with finite support. Denote \(\ChargeDensitySet\) the set of charge densities, and \(\ChargeDensitySet_{\Lambda}\) the set of charge densities with support included in \(\Lambda\).
    \item A \emph{charge ensemble} \(\Ecal\) is a finite (possibly empty) set of charge densities, with the property that if \(\rho,\varrho\in \Ecal\) and \(\rho \neq \varrho\), then \(\supp \rho \cap \supp \varrho = \varnothing\) (disjoint supports).
    \item The \emph{total charge} of a charge density \(\rho\) is \(\rho \cdot 1 = \sum_{i} \rho(i)\). \(\rho\) is said to be \emph{neutral} if \(\rho \cdot 1 = 0\). When \(\rho \cdot 1 \neq 0\), \(\rho\) is said to be \emph{charged} or \emph{non-neutral}.
    \item For \(\rho,\varrho\in \ChargeDensitySet\), say that \(\varrho\) is a \emph{sub-density} of \(\rho\), denoted \(\varrho \subset \rho\) if \(\supp \varrho \subset \supp \rho\) and \(\rho -\varrho\) is \(0\) on \(\supp \varrho\).
\end{itemize}

\subsubsection*{Square covering of densities}

For \(k\geq 0\), let \(\SquareSet_k\) be the set of \(2^k\times 2^k\) squares:
\begin{equation*}
    \SquareSet_k = \big\{x+\{0,\dots, 2^{k}-1\}^2:\ x\in \Z^2\big\}.
\end{equation*}

For \(\Lambda\subset \Z^2\) fixed, and for a charge density \(\rho \in \ChargeDensitySet_{\Lambda}\), define its \(k\)-\emph{square covering}, \(\SquareCover_k(\rho)\) to be a set of squares \(\{S_1,\dots,S_m\} \subset \SquareSet_k\) such that
\begin{enumerate}
    \item \(\supp \rho \subset \cup_{l=1}^m S_l\);
    \item \(m\) is the smallest integer for which the first point holds.
\end{enumerate}Note that there are in general several choices for a cover satisfying the two above points. As it is only used in the derivation of results we import from~\cite{KP17,W19}, we make the same choice as in these works. Note that \(\SquareCover_k(\rho)\) depends only on \(\supp \rho\), so we extend the notation to \(\SquareCover_k(W)\) for \(W\subset \Z^2\) finite by setting \(\SquareCover_k(W) \equiv \SquareCover_k(\mathds{1}_W)\), with \(\mathds{1}_W\) the charge density equals to \(1\) on \(W\) and \(0\) else.

Define the \emph{centre} of a charge density \(\rho\), \(\DensCenter(\rho)\), to be the (chosen by some fixed rule in case of several choices) site \(v_*\) such that there is \(j\in \supp \rho\) with \(\dist(v_*,j) = \diam \rho\) when \(\dd_{\Lambda}(\rho) = \diam \rho\), and such that \(\dist(v_*,\Lambda^c) = \dd_{\Lambda}(\rho)\) otherwise. Define the \emph{envelope} of a charge density \(\rho\) by
\begin{equation}
    \label{eq:def:Envelope}
    \DensSquareEnv_{\Lambda}(\rho) = \{i\in \Z^2:\ \dist(\DensCenter(\rho),i)< 2\dd_{\Lambda}(\rho)  \},
\end{equation}and say that \(\DensCenter_{\Lambda}(\rho)\) is the \emph{centre} of \(\DensSquareEnv_{\Lambda}(\rho)\). Let
\begin{equation}
    \label{eq:def:Large_Envelope}
    \DensSquareEnv_{\Lambda}^+(\rho) = \{i\in \Z^2:\ \dist(i,\DensSquareEnv_{\Lambda}(\rho))\leq 1\}.
\end{equation}
For \(\Lambda\subset \Z^2\), define also:
\begin{itemize}
    \item the \emph{scale} of \(\rho \in \ChargeDensitySet\),
    \begin{equation}
        \scale_{\Lambda}(\rho) := \lceil \log_2(M\dd_{\Lambda}(\rho)^{\alpha} )\rceil;
    \end{equation}
    \item for \(\rho \in \ChargeDensitySet\),
    \begin{equation}
        A_{\Lambda}(\rho) :=
        \sum_{k=0}^{\scale_{\Lambda}(\rho)} |\SquareCover_k(\rho)|;
    \end{equation}
    \item the set of \emph{well separated} squares in the \(k\)-cover of \(\rho\in \ChargeDensitySet\): define \(\SquareCover^{\sep}_k(\rho)\) by
    \begin{itemize}
        \item if \(|\SquareCover_k(\rho)|>1\):
        \begin{equation*}
            \SquareCover^{\sep}_k(\rho) = \big\{ s\in \SquareCover_k(\rho):\ \min_{s'\neq s}\dist(s,s')\geq 2M2^{\alpha(k+1)} \big\}
        \end{equation*}where the \(\min\) is over \(s'\in \SquareCover_k(\rho)\);
        \item if \(|\SquareCover_k(\rho)| = 1\) and \(\rho\cdot 1 \neq 0\) and \(\dist(\rho, \Lambda^c)\geq 2^{k+1}\), \(\SquareCover^{\sep}_k(\rho) = \SquareCover_k(\rho)\);
        \item else \(\SquareCover^{\sep}_k(\rho) = \varnothing\).
    \end{itemize}
\end{itemize}

One has a control over \(A(\rho)\) in terms of \(\SquareCover^{\sep}_k(\rho)\) and \(\dd_{\Lambda}(\rho)\). This is~\cite[Proposition 20]{W19} (which is an adaptation of \cite[Proposition 2.1]{KP17}).

\begin{lemma}[\cite{W19}, Proposition 20]
    \label{lem:control_A_with_Ssep}
    There is \(\Cl{A_Lambda}> 0\) such that for any \(\Lambda\subset \Z^2\), \(\rho\in \ChargeDensitySet_{\Lambda}\),
    \begin{equation*}
        \log_2(1+\dd_{\Lambda}(\rho) ) \leq A_{\Lambda}(\rho)\leq \Cr{A_Lambda} \Big(|\SquareCover_0(\rho)| + \sum_{k=1}^{\scale_{\Lambda}(\rho)} |\SquareCover_{k}^{\sep}(\rho)|\Big).
    \end{equation*}
\end{lemma}

\subsection{Multipole expansion}

Next we restate a core part of the multiscale argument of \cite{FS81a}, an expansion of the trigonometric density as a linear combination over charge ensembles of constraints in terms of the charges in that ensemble.

\begin{theorem}
\label{thm:convex_combination_ensembles}
Let \(\Lambda\subset \Z^2\) be a finite connected set. Let \(\lambda_i\in \TrigoPolySet\), \(i\in \Lambda\). There exists a finite collection of ensembles $\Fcal$, positive coefficients $(c_{\Ncal})_{\Ncal \in \Fcal}$ summing to 1 and real numbers $(K(\rho))_{\rho \in \Ncal, \Ncal \in \Fcal}$, such that for every $\phi\colon \Lambda \to \bbC$,
\begin{equation}
\label{e:convex_combination_ensembles}
    \lambda_{\Lambda}(\phi) = \sum_{\Ncal \in \Fcal} c_\Ncal \prod_{\rho \in \Ncal}\left(1 + 2K(\rho) \cos(\phi\cdot\rho)\right).
\end{equation}
Additionally, the following properties are satisfied for each $\Ncal \in \Fcal$:
\begin{enumerate}[label=(\roman*)]
    \item\label{i:distinctchargessep} If $\rho, \rho' \in \Ncal$ are distinct then $\dist(\rho,\rho') \ge M \left(\min(\dd_{\Lambda}(\rho), \dd_{\Lambda}(\rho'))\right)^\alpha$.
    \item\label{i:isolateddensitycharged} If $\rho_1 \subset \rho \in \Ncal$, $\rho_1 \neq \rho$, satisfies $\dist(\rho_1, \rho-\rho_1) \ge 2M \diam(\rho_1)^\alpha$, then $\rho_1\cdot 1 \neq 0$ and
    \begin{equation*}
        2M \dist(\rho_1, \Lambda^c)^\alpha > \dist(\rho_1, \rho-\rho_1).
    \end{equation*}
    \item\label{i:boundcoeffdecomp} The coefficients $K(\rho)$ satisfy 
    \begin{equation*}
        2|K(\rho)|\le \exp(\Cl{entropy_ensemble} A_{\Lambda}(\rho))\prod_{v \in \supp(\rho)} 2\exp(|\rho_v|)\big|\hat\lambda_v(|\rho_v|)\big|,
    \end{equation*}
    for some (universal) constant $\Cr{entropy_ensemble}$.
\end{enumerate}
\end{theorem}
\begin{proof}
Properties \ref{i:distinctchargessep}, \ref{i:isolateddensitycharged} and \ref{i:boundcoeffdecomp} are essentially \cite[Theorem 21]{W19}. The only difference is that in the first step of the proof we do not use the weights $e^{q^2}$, but $e^q$ (which might lead to an additional factor of 2).
\end{proof}

From this result, we obtain the following:
\begin{corollary}
\label{cor:multipole_exp_unnorm_correlations}
    Let \(\Lambda\subset \Z^2\) be a finite connected set. Let \(\lambda_i\in \TrigoPolySet\), \(i\in \Lambda\). Then, there are \(\Fcal\), \((c_{\Ncal})_{\Ncal\in \Fcal}\), \(K(\rho),\rho\in\Ncal\in \Fcal\) as in Theorem~\ref{thm:convex_combination_ensembles} such that for any \(f,\sigma,\zeta \colon \Lambda\to \R\), one has
    \begin{multline*}
        \int d\phi \calI_{\Lambda}^{\xi}(\phi) e^{f\cdot \phi} \lambda_{\Lambda}(\phi- \zeta)
        \\=
        \sum_{\Ncal\in \Fcal}c_{\Ncal} \underbrace{e^{\sigma\cdot f}\int d\phi \calI_{\Lambda}^{\xi}(\phi+\sigma) e^{f\cdot \phi} \prod_{\rho\in \Ncal}\big(1 + 2K(\rho)\cos(\phi\cdot \rho - \zeta\cdot \rho + \sigma \cdot \rho)\big)}_{=: Z_{\Ncal}^{\zeta}(f;\sigma)}.
    \end{multline*}
\end{corollary}
\begin{proof}
    Use the change of variable \(\phi\mapsto \phi + \sigma\) and apply Theorem~\ref{thm:convex_combination_ensembles}.
\end{proof}
Now, if we denote
\begin{equation*}
    Z_{\Ncal}^{\zeta} := Z_{\Ncal}^{\zeta}(0;0).
\end{equation*}
we have obtained the representation
\begin{equation}
    \label{eq:multipol_exp:correlation_I}
    \int d\mu_{\Lambda}^{\xi} e^{f\cdot \phi}\lambda_{\Lambda}(\phi)
    \propto
    \sum_{\Ncal\in \Fcal}c_{\Ncal} Z_{\Ncal}^{\zeta}(f;\sigma),
\end{equation}where the proportionality constant does not depend on \(f,\sigma\).

\subsection{Modified measure}

Starting from the combinatorial expansion of the previous section, we will, as in~\cite{FS81a}, ``renormalize'' the activities of each charge density. This will lead to a representation of \(Z_{\Ncal}^{\zeta}(f;\sigma)\) as the expectation of a suitable observable under a suitable positive, finite, measure. This measure will by defined using an alternative weight function. We define it here and state its properties, and will recall the definition at the relevant place.
For \(x,a\in \R\), let
\begin{equation}
    \label{eq:def:iota_beta}
    \iota_{\beta}(x,a) = \frac{I_{\beta}(x + \rmi a)}{I_{\beta}(x)}e^{-(2\Cr{lnIratio_deriv}+1)c_{\beta}g(a)}
\end{equation}
For \(V\subset \Z^2\) finite, let \(E_V = \{\{i,j\}:\ i\in V, j\sim i \}\), and set
\begin{equation}
\label{eq:def:iota_V}
    \iota_{V}(\phi, a)\equiv \iota_{V;\beta}(\phi, a) = \prod_{\{i,j\}\in E_V} \iota_{\beta}(\phi_i-\phi_j,a_i-a_j).
\end{equation}

\begin{lemma}
    \label{lem:iota_prop_derivatives}
   In the setting of Subsection~\ref{subsec:models} the functions \(\iota_{\beta}\), \(\iota_{V}\) satisfy
    \begin{enumerate}
        \item \(|\iota_{\beta}(x,a)|\leq e^{-2\Cr{lnIratio_deriv}c_{\beta}g(a) }\), thus \(|\iota_{V}(\phi, a)|\leq \prod_{i\sim j\in V}e^{-2\Cr{lnIratio_deriv}c_{\beta}g(a_i-a_j) }\);
        \item for \(a: \Z^2 \to \R\) with \(\sup_{i\sim j}|a_i-a_j|< \epsilon_{\beta}\), and \(\{i,j\},\{u,v\}\in E_{V}\),
        \begin{gather*}
            \Big|\frac{\dd}{\dd (\phi_i-\phi_j)} \iota_{V}(\phi, a)\Big|
            \leq 
            1+\Cr{lnIratio_deriv}c_{\beta}
            \\
            \Big|\frac{\dd^2}{\dd (\phi_i-\phi_j)\dd (\phi_u-\phi_v)} \iota_{V}(\phi, a)\Big|
            \leq
            (\Cr{lnIratio_deriv}c_{\beta}+1)(\Cr{lnIratio_deriv}c_{\beta}+2).
        \end{gather*}
    \end{enumerate}
\end{lemma}
\begin{proof}
    The first point follows from the definition and Property~\ref{item:Ibeta:imaginary_growth} of \(I_{\beta}\). To get the second point, note that, writing \(x_{ij} \equiv x_i -x_j\) and using \(f' = f(\ln f)'\),
    \begin{align*}
        \frac{\dd}{\dd \phi_{ij}} \iota_V(\phi,a)
        &=
        \iota_V(\phi,a) \frac{I_{\beta}(\phi_{ij})}{I_{\beta}(\phi_{ij} + \rmi a_{ij})}\frac{\dd}{\dd \phi_{ij}} \frac{I_{\beta}(\phi_{ij}+ \rmi a_{ij})}{I_{\beta}(\phi_{ij} )}
        \\
        &=
        \iota_V(\phi,a) \frac{\dd}{\dd \phi_{ij}} \ln\Big(\frac{I_{\beta}(\phi_{ij}+ \rmi a_{ij})}{I_{\beta}(\phi_{ij} )}\Big).
    \end{align*}
    So, from the first point and Property~\ref{item:Ibeta:derivatives} of \(I_{\beta}\),
    \begin{equation*}
        \Big|\frac{\dd}{\dd \phi_{ij}} \iota_V(\phi,a)\Big|
        \leq
        \Cr{lnIratio_deriv}c_{\beta} (1+g(a_i-a_j))e^{-2\Cr{lnIratio_deriv}c_{\beta}g(a_i-a_j) }
        \leq
        \Cr{lnIratio_deriv}c_{\beta}+1,
    \end{equation*}as \(xe^{-x}\leq 1\) for \(x\geq 0\). The bound on the second derivative follows in the same fashion.
\end{proof}

\subsection{Activity renormalization}
\label{subsec:activity_renorm}

We now renormalize the activity of the charge densities in \(Z_{\Ncal}^{\zeta}(f;\sigma)\), compensating the entropy of a charge density, \(K(\rho)\), by some small activity. The activity renormalization is performed by mean of complex translations, exactly as in~\cite{FS81a}. The quantity measuring the energy gain is given by
\begin{equation}
\label{eq:def:E_beta}
    E_{\beta}(\rho, a) = \rho \cdot a - (2\Cr{lnIratio_deriv}+1)c_{\beta} \sum_{i\sim j} g(a_i-a_j)
\end{equation}where \(a\) is a \emph{spin-wave} (the value of the complex translation). The first term is the bare energy gain, whilst the second is the loss due to the choice of normalization in the \(\iota_{V}\). The spin waves used in this step are chosen to have large \(E_{\beta}\), and are given by the next Lemma.

\begin{lemma}
    \label{lem:spin_waves}
In the setting of Subsection~\ref{subsec:models} suppose we are given some $\gamma_\beta\in(0,\epsilon_\beta$. Then the following holds. For any \(\Lambda\subset \Z^2\) finite and connected, and \(\Ncal\) a charge ensemble satisfying the properties given by Theorem~\ref{thm:convex_combination_ensembles}, one can construct spin waves \(a_{\Ncal,\rho}:\Lambda\to \R\), \(\rho\in \Ncal\) such that for any \(\rho,\rho'\in \Ncal\),
    \begin{enumerate}
        \item \label{spinWave:support}\(\supp a_{\Ncal,\rho}\subset \DensSquareEnv_{\Lambda}(\rho)\cap \Lambda\);
        \item \label{spinWave:orthogonality} when \(\rho\neq \rho'\), one has that \(\supp \nabla a_{\Ncal,\rho}\cap \supp \nabla a_{\Ncal,\rho'} = \varnothing\), and \(a_{\Ncal,\rho}\cdot \rho' =0\);
        \item \label{spinWave:energy} There are universal constants $\Cl[c]{gamma}>0$, $\Cl[c]{energy}>0$ such that $\frac{c_{\beta}g(\gamma_\beta)}{\gamma_\beta}\le \frac{\Cr{gamma}}{\Cr{g_quadr}}$ implies that
        \[E_{\beta}(a_{\rho},\rho)\geq \Cr{energy}\gamma_\beta\Big(\|\rho\|_1+A_\Lambda(\rho)\Big);\]
        \item \label{spinWave:gradient}for any \(i\sim j\), \(|a_{\Ncal,\rho}(i)-a_{\Ncal,\rho}(j)|\leq \gamma_{\beta} \);
        \item \label{spinWave:scalar_prod} \(\norm{\nabla a_{\Ncal,\rho}}_2^2 \leq
        \gamma_{\beta}^2 \Cl{norm_gradSpinWave} \big(A_{\Lambda}(\rho)+\|\supp(\rho)\|_{1}\big)\) where \(\Cr{norm_gradSpinWave}\) is a universal constant.
    \end{enumerate}
\end{lemma}
This lemma is essentially proven in \cite{FS81a}, yet for completeness we give a proof in Appendix~\ref{sec:spin_waves}. 

\noindent We then use these spin waves to renormalize the activities in the expansion given by Corollary~\ref{cor:multipole_exp_unnorm_correlations}. As we will mostly work with \(\Ncal\) fixed, and we will use these objects several times in the equations, we will frequently use the shorthands
\begin{equation}
    \label{eq:shorthand_spin_waves}
    a_{\rho} \equiv a_{\Ncal,\rho},\qquad V_{\rho} \equiv \supp \nabla a_{\rho}.
\end{equation}
Using Lemma~\ref{lem:complex_translations}, one has that \(e^{-\sigma\cdot f}Z_{\Ncal}^{\zeta}(f,\sigma)\) is equal to
\begin{multline*}
    \sum_{n\in \{-1,0,1\}^{\Ncal}}\int d\phi \calI_{\Lambda}^{\xi}(\phi+\sigma) e^{f\cdot \phi} \prod_{\rho\in \Ncal}K(\rho)^{|n_{\rho}|}e^{\rmi (\phi\cdot \rho- \zeta\cdot \rho + \sigma \cdot \rho)n_{\rho}}
    \\
    =
    \sum_{n\in \{-1,0,1\}^{\Ncal}}\int d\phi \calI_{\Lambda}^{\xi}(\phi+\sigma + \rmi a_n) e^{f\cdot \phi}e^{\rmi f\cdot a_n} \prod_{\rho\in \Ncal}\big(e^{-a_{\Ncal,\rho}\cdot \rho}K(\rho)\big)^{|n_{\rho}|}e^{\rmi (\phi\cdot \rho - \zeta\cdot \rho + \sigma \cdot \rho)n_{\rho}},
\end{multline*}where
\begin{equation*}
    a_n = \sum_{\rho\in \Ncal} n_{\rho} a_{\Ncal,\rho},
\end{equation*}and we used that by construction \(a_{\Ncal,\rho}\cdot \rho' = 0\) when \(\rho \neq \rho'\).

Now, recall the weights \(\iota_{V}(\phi, a)\) from~\eqref{eq:def:iota_V}, and its properties (Lemma~\ref{lem:iota_prop_derivatives}). In particular, \(|\iota_{V}(\phi, a)|\leq 1\). As \(V_{\rho},V_{\rho'}\) are disjoint for \(\rho\neq \rho'\) (see Lemma~\ref{lem:spin_waves}, item~\ref{spinWave:orthogonality}), using the shorthand~\eqref{eq:shorthand_spin_waves}, one has
\begin{equation*}
    \frac{\calI_{\Lambda}^{\xi}(\phi+\sigma + \rmi a_n)}{\calI_{\Lambda}^{\xi}(\phi +\sigma)}
    =
    \prod_{\rho \in \Ncal} \Big(\prod_{i\sim j: \{i,j\}\cap V_{\rho}\neq \varnothing} e^{(1+2\Cr{lnIratio_deriv})g_{\beta}(a_{\rho}(i)-a_{\rho}(j))}\Big)\iota_{V_{\rho}}(\phi+\sigma, n_{\rho}a_{\Ncal,\rho}).
\end{equation*}Plugging this in the previous expression for \(e^{-\sigma\cdot f}Z_{\Ncal}^{\zeta}(f,\sigma)\) and re-exchanging summation over \(n\) and integral over \(\phi\), one obtains that \(Z_{\Ncal}^{\zeta}(f,\sigma)\) is equal to
\begin{multline*}
    e^{\sigma\cdot f}\int d\phi \calI_{\Lambda}^{\xi}(\phi+\sigma) e^{f\cdot \phi} \prod_{\rho\in \Ncal} \Big(1
    +
    z_{\rho} \Re(\iota_{V_{\rho}}(\phi+\sigma, a_{\Ncal,\rho}))\cos((\phi-\zeta +\sigma)\cdot \rho + f\cdot a_{\Ncal,\rho}) 
    \\
    -
    z_{\rho}\Im(\iota_{V_{\rho}}(\phi+\sigma, a_{\Ncal,\rho})) \sin((\phi-\zeta +\sigma)\cdot \rho + f\cdot a_{\Ncal,\rho})\Big),
\end{multline*}
where
\begin{equation*}
    z_{\rho} = 2e^{-a_{\Ncal,\rho}\cdot \rho}K(\rho)\prod_{i\sim j: \{i,j\}\cap V_{\rho}\neq \varnothing} e^{(1+2\Cr{lnIratio_deriv})g_{\beta}(a_{\rho}(i)-a_{\rho}(j))}
    \equiv 2K(\rho) e^{-E_{\beta}(a_{\Ncal,\rho}, \rho)},
\end{equation*}and we used that, as \(I_{\beta}\) is real on the reals, \(\iota_{V}(\phi, -a) = \overline{\iota_{V}(\phi,a)}\). Now, notice that by Theorem~\ref{thm:convex_combination_ensembles} item \ref{i:boundcoeffdecomp}, and Lemma~\ref{lem:spin_waves}, item~\ref{spinWave:energy}, one has, if $\gamma_\beta$ is larger than some (absolute) constant,
\begin{equation}
    \label{eq:renorm_activity_upper_bound}
        |z_{\rho}|
        \leq
        e^{-\Cr{energy}\gamma_\beta(\|\rho\|_1+A_\Lambda(\rho))} \prod_{i\in \supp \rho} 2e^{|\rho_i|}|\hat{\lambda}(|\rho_i|)|
        \leq
        e^{-\Cr{energy}\gamma_\beta\norm{\rho}_{1}/2
        - \Cr{energy}\gamma_\beta A_\Lambda(\rho)},
\end{equation}
Here we used \(|\hat{\lambda}|\leq 1\) and Lemma \ref{lem:control_A_with_Ssep} in the last step. 

In particular, one has that \(|z_{\rho}|\leq 8^{-1}\) uniformly in $\rho$ for any \(\gamma_\beta\) large enough, so
\begin{multline*}
    1
    +
    z_{\rho} \Re(\iota_{V_{\rho}}(\phi+\sigma, a_{\Ncal,\rho}))\cos((\phi - \zeta + \sigma) \cdot \rho + f\cdot a_{\Ncal,\rho})
    \\
    -
    z_{\rho}\Im(\iota_{V_{\rho}}(\phi+\sigma, a_{\Ncal,\rho})) \sin((\phi - \zeta + \sigma)\cdot \rho + f\cdot a_{\Ncal,\rho})>0.
\end{multline*}In particular, \(Z_{\Ncal}^{\zeta}>0\), so letting
\begin{equation*}
    c'_{\Ncal} = \frac{c_{\Ncal}Z_{\Ncal}^{\zeta}}{\sum_{\Ncal\in \Fcal} c_{\Ncal} Z_{\Ncal}^{\zeta}},
\end{equation*}one has \(c'_{\Ncal}>0\), \(\sum_{\Ncal\in \Fcal} c'_{\Ncal} = 1\), and~\eqref{eq:multipol_exp:correlation_I} gives
\begin{equation}
    \label{eq:multipol_exp:correlation_II}
    \frac{\int d\mu_{\Lambda}^{\xi}(\phi) e^{f\cdot \phi}\lambda_{\Lambda}(\phi-\zeta)}{\int d\mu_{\Lambda}^{\xi}(\varphi) \lambda_{\Lambda}(\varphi-\zeta)}
    =
    \sum_{\Ncal\in \Fcal}c_{\Ncal}' \frac{Z_{\Ncal}^{\zeta}(f;\sigma) }{Z_{\Ncal}^{\zeta}}.
\end{equation}It remains to find lower bounds on~\eqref{eq:multipol_exp:correlation_II}. To this end, we need to be more careful than what is done in~\cite{FS81a}. This is the content of the next section (and the main contribution of this work).

\section{Improved treatment of the renormalized model}\label{sec:improvements}

After all the prelimaries from Section~\ref{sec:chargesrenorm}, we can now proceed to our main contribution, the proofs of Theorem~\ref{thm:LB_trigo_poly_non_sym} and \ref{thm:LB_trigo_poly_sym}.

Throughout this section, we will assume that we are in the setting of Subsection~\ref{subsec:models}, and that we have \eqref{eq:def:gamma_beta} for some $\Cr{lowerboundmainthm}>0$ and some sufficiently small positive constant $\Cr{quantsmallgamma}>0$. By choosing $\Cr{quantsmallgamma}$ small enough, this in particular ensures that $\gamma_\beta$ is large enough that the derivation of \eqref{eq:multipol_exp:correlation_II} in the previous subsection is valid.

\subsection{Taylor expansions}
\label{subsec:Taylor}

The main step in the proof of the theorems is to provide lower bounds on ratios \(\frac{Z_{\Ncal}^{\zeta}(f;\sigma) }{Z_{\Ncal}^{\zeta}}\) obtained in Subsection~\ref{subsec:activity_renorm}. Like in \cite{FS81a}, we will do so by Taylor-expanding the numerator to second order. The main difference to \cite{FS81a} is that we will do this expansion not just in $\sigma\cdot\rho$ but also in $f\cdot a_{\Ncal,\rho}$, which allows us to avoid the loss of a multiplicative constant. As the details are very technical, we first give the full expansion and then explain in Remark~\ref{r:techimprovement} how exactly we deviate from \cite{FS81a}.

 From now on, fix \(\Ncal\) as in~\eqref{eq:multipol_exp:correlation_II}. For \(\rho\in \Ncal\), introduce the shorthands
\begin{gather*}
    F_{\rho}(\phi,\sigma,f) = \Re(\iota_{V_{\rho}}(\phi+\sigma, a_{\Ncal,\rho}))\cos( (\phi-\zeta + \sigma) \cdot \rho + f\cdot a_{\Ncal,\rho}),
    \\
    G_{\rho}(\phi,\sigma,f) = \Im(\iota_{V_{\rho}}(\phi+\sigma, a_{\Ncal,\rho})) \sin((\phi-\zeta + \sigma) \cdot \rho + f\cdot a_{\Ncal,\rho}),
\end{gather*}
and denote \(F_{\rho}(\phi) \equiv F_{\rho}(\phi,0,0)\), and \(G_{\rho}(\phi) \equiv G_{\rho}(\phi,0,0)\). Also recall that if $\gamma_\beta$ is large enough, then for any \(\rho\in \Ncal\),
\begin{equation}
    |z_{\rho}| \leq \frac{1}{8},\quad |F_{\rho}(\phi,\sigma,f)|\leq 1,\quad |G_{\rho}(\phi,\sigma,f)|\leq 1.
\end{equation}
We will build on the (trivial) re-writing:
\begin{equation}
\label{eq:integral_rep_ratio_ensemble_PF}
    \begin{aligned}
        \frac{Z_{\Ncal}^{\zeta}(f;\sigma) }{Z_{\Ncal}^{\zeta}}
        &=
        e^{\sigma\cdot f}\frac{\int d\phi \calI_{\Lambda}^{\xi}(\phi+\sigma) e^{f\cdot \phi} \prod_{\rho\in \Ncal} \Big(1
        +
        z_{\rho} F_{\rho}(\phi,\sigma,f)
        -
        z_{\rho}G_{\rho}(\phi,\sigma,f)\Big)}{\int d\phi \calI_{\Lambda}^{\xi}(\phi) \prod_{\rho\in \Ncal} \Big(1
        +
        z_{\rho} F_{\rho}(\phi)
        -
        z_{\rho}G_{\rho}(\phi)\Big)}
        \\
        &=
        e^{\sigma\cdot f}\int d\nu_{\Lambda}^{\xi}(\phi) e^{f\cdot \phi}\frac{\calI_{\Lambda}^{\xi}(\phi+\sigma)}{\calI_{\Lambda}^{\xi}(\phi)} \prod_{\rho\in \Ncal} \frac{1+ z_{\rho}F_{\rho}(\phi,\sigma,f) - z_{\rho}G_{\rho}(\phi,\sigma,f)}{1+ z_{\rho}F_{\rho}(\phi) - z_{\rho}G_{\rho}(\phi)},
    \end{aligned}
\end{equation}where \(\nu_{\Lambda}^{\xi}\) is the probability measure on \(\R^{\Lambda}\) given by
\begin{equation}
\label{eq:nu_Lambda}
    d\nu_{\Lambda}^{\xi}(\phi) = \frac{\calI_{\Lambda}^{\xi}(\phi) \prod_{\rho\in \Ncal} \big(1
    +
    z_{\rho} F_{\rho}(\phi)
    -
    z_{\rho}G_{\rho}(\phi)\big)}{\int d\varphi \calI_{\Lambda}^{\xi}(\varphi) \prod_{\rho\in \Ncal} \big(1
    +
    z_{\rho} F_{\rho}(\varphi)
    -
    z_{\rho}G_{\rho}(\varphi)\big)} d\phi.
\end{equation}The first step is to Taylor-expand around \(\phi\) the two types of factors appearing in the last integral.
\begin{claim}
\label{claim:Taylor_FG}
    Using the shorthands~\eqref{eq:shorthand_spin_waves}, for any \(\rho\in \Ncal\in \Fcal\) as in Theorem~\ref{thm:convex_combination_ensembles}, and for any $\gamma_\beta$ large enough, we have
    \begin{equation*}
        \ln\Bigg(\frac{1+ z_{\rho}F_{\rho}(\phi,\sigma,f) - z_{\rho}G_{\rho}(\phi,\sigma,f)}{1+ z_{\rho}F_{\rho}(\phi) - z_{\rho}G_{\rho}(\phi)}\Bigg)
        =
        S_{\rho}(\phi,\sigma,f) + r_{\rho}(\phi,\sigma,f),
    \end{equation*}where
    \begin{align*}
        S_{\rho}(\phi,\sigma,f)
        &=
        \frac{z_{\rho}}{1+ z_{\rho}F_{\rho}(\phi) - z_{\rho}G_{\rho}(\phi)}
        \\
        &\quad\times
        \Big(\cos((\phi+\zeta)\cdot \rho) \nabla\sigma \cdot \nabla_{\nabla \phi} \Re(\iota_{V_{\rho}}(\phi, a_{\rho}))
        \\
        &\quad\phantom{\times
        \Big(} +
        \sin((\phi+\zeta)\cdot \rho) \nabla\sigma \cdot \nabla_{\nabla \phi} \Im(\iota_{V_{\rho}}(\phi, a_{\rho}))
        \\
        &\quad\phantom{\times
        \Big(} +
        \cos((\phi+\zeta)\cdot \rho)(\sigma \cdot \rho + f\cdot a_{\rho})\Re\big(\iota_{V_{\rho}}(\phi, a_{\rho})\big)
        \\
        &\quad\phantom{\times
        \Big(} -
        \sin((\phi+\zeta)\cdot \rho)(\sigma \cdot \rho + f\cdot a_{\rho})\Im\big(\iota_{V_{\rho}}(\phi, a_{\rho})\big)\Big),
    \end{align*}
    and
    \[
        |r_{\rho}(\phi,\sigma,f)|
        \leq
        6|z_\rho||\sigma\cdot \rho + a_{\rho}\cdot f|^2
        +24(\Cr{lnIratio_deriv}c_{\beta}+2)^2|z_\rho| |V_{\rho}|\norm{(\nabla \sigma)\restrict_{V_{\rho}}}_2^2
    \]uniformly over \(\phi\).
\end{claim}
\begin{proof}
For readability, we deal with \(\zeta  = 0\). The procedure for \(\zeta\neq 0\) is exactly the same. Recall the shorthands~\eqref{eq:shorthand_spin_waves}. One first has the expansions (by zeroth and first order Taylor-Lagrange)
    \begin{align*}
        \cos(\phi\cdot \rho + \sigma \cdot \rho + f\cdot a_{\rho}) &= \cos(\phi\cdot \rho) +r_{\cos}^{(1)}(\phi,\sigma,f),\\
        \cos(\phi\cdot \rho + \sigma \cdot \rho + f\cdot a_{\rho}) &= \cos(\phi\cdot \rho) - \sin(\phi\cdot \rho)(\sigma \cdot \rho + f\cdot a_{\rho}) + r_{\cos}^{(2)}(\phi,\sigma,f),
        \\
        \sin(\phi\cdot \rho + \sigma \cdot \rho + f\cdot a_{\rho}) &= \sin(\phi\cdot \rho) + r_{\sin}^{(1)}(\phi,\sigma,f),\\
        \sin(\phi\cdot \rho + \sigma \cdot \rho + f\cdot a_{\rho}) &= \sin(\phi\cdot \rho) + \cos(\phi\cdot \rho) (\sigma \cdot \rho + f\cdot a_{\rho}) + r_{\sin}^{(2)}(\phi,\sigma,f),
    \end{align*}with 
    \begin{align*}
    |r_{\cos}^{(1)}(\phi,\sigma,f)|,|r_{\sin}^{(1)}(\phi,\sigma,f)|&\leq |\sigma \cdot \rho + f\cdot a_{\rho}|\\
    |r_{\cos}^{(2)}(\phi,\sigma,f)|,|r_{\sin}^{(2)}(\phi,\sigma,f)|&\leq \frac12|\sigma \cdot \rho + f\cdot a_{\rho}|^2.
    \end{align*}Similarly, from the properties of \(I_{\beta}\), one has (by zeroth and first order multivariate Taylor-Lagrange)
    \begin{align*}
        \Re(\iota_{V_{\rho}}(\phi+\sigma, a_{\rho}))
        &=
        \Re(\iota_{V_{\rho}}(\phi, a_{\rho}))
        + r_{\Re}^{(1)}(\phi,\sigma),
        \\
        \Re(\iota_{V_{\rho}}(\phi+\sigma, a_{\rho}))
        &=
        \Re(\iota_{V_{\rho}}(\phi, a_{\rho}))
        + \nabla\sigma \cdot \nabla_{\nabla \phi} \Re(\iota_{V_{\rho}}(\phi, a_{\rho}))
        + r_{\Re}^{(2)}(\phi,\sigma),
        \\
        \Im(\iota_{V_{\rho}}(\phi+\sigma, a_{\rho}))
        &=
        \Im(\iota_{V_{\rho}}(\phi, a_{\rho}))
        + r_{\Im}^{(1)}(\phi,\sigma),
        \\
        \Im(\iota_{V_{\rho}}(\phi+\sigma, a_{\rho}))
        &=
        \Im(\iota_{V_{\rho}}(\phi, a_{\rho}))
        + \nabla\sigma \cdot \nabla_{\nabla \phi} \Im(\iota_{V_{\rho}}(\phi, a_{\rho}))
        + r_{\Im}^{(2)}(\phi,\sigma),
    \end{align*}where \(\nabla_{\nabla \phi}\) the vector of partial derivatives with respect to the gradient of \(\phi\): for \(\sigma\) such that \(\nabla\sigma\) has finite support
    \begin{equation*}
        \nabla\sigma \cdot \nabla_{\nabla \phi} h(\phi) = \sum_{i\sim j}(\sigma_i-\sigma_j)\frac{\dd}{\dd(\phi_i-\phi_j)} h(\phi).
    \end{equation*}By Lemma~\ref{lem:iota_prop_derivatives}, and the fact that \(\frac{\dd}{\dd (\phi_i-\phi_j)} \iota_{V_{\rho}} = 0\) when \(\{i,j\}\cap \supp \nabla a_{\rho} = \varnothing\), \(r_{\Im}(\phi,\sigma)\) and \(r_{\Re}(\phi,\sigma)\) satisfy
    \begin{align*}
        |r_{\Re}^{(1)}(\phi,\sigma)|,|r_{\Im}^{(1)}(\phi,\sigma)| &\leq 
        (1+\Cr{lnIratio_deriv}c_{\beta})\norm{(\nabla \sigma)\restrict_{V_\rho}}_{1}\le (1+\Cr{lnIratio_deriv}c_{\beta})(4|V_{\rho}|)^{1/2}\norm{(\nabla \sigma)\restrict_{V_{\rho}}}_{2},
        \\
        |r_{\Re}^{(2)}(\phi,\sigma)|,|r_{\Im}^{(2)}(\phi,\sigma)| &\leq 
        \frac12\cdot4(2+\Cr{lnIratio_deriv}c_{\beta})^2|V_{\rho}|\norm{(\nabla \sigma)\restrict_{V_{\rho}}}_{2}^2,
    \end{align*}
    where a factor \(4\) comes in as there are at most \(4\) edges adjacent to a site in \(\supp \nabla a_{\rho}\).
    
    Putting these estimates together and using \(|\iota_V|\leq 1\), and Lemma~\ref{lem:iota_prop_derivatives}, one gets the first order expansions
    \begin{align*}
        F_{\rho}(\phi,\sigma,f)
        &=
        F_{\rho}(\phi)
        +
        \cos(\phi\cdot \rho) \nabla\sigma \cdot \nabla_{\nabla \phi} \Re(\iota_{V_{\rho}}(\phi, a_{\rho}))
        \\
        &\qquad\qquad-
        \sin(\phi\cdot \rho)(\sigma \cdot \rho + f\cdot a_{\rho})\Im\big(\iota_{V_{\rho}}(\phi, a_{\rho})\big)\\
        &\qquad\qquad+\Re(\iota_{V_{\rho}}(\phi, a_{\rho}))r_{\cos}^{(2)}(\phi,\sigma,f)+r_{\Re}^{(1)}(\phi,\sigma)r_{\cos}^{(1)}(\phi,\sigma,f)+r_{\Re}^{(2)}(\phi,\sigma)\cos( \phi\cdot \rho )\\
        &=
        F_{\rho}(\phi)
        +
        \cos(\phi\cdot \rho) \nabla\sigma \cdot \nabla_{\nabla \phi} \Re(\iota_{V_{\rho}}(\phi, a_{\rho}))
        \\
        &\qquad\qquad-
        \sin(\phi\cdot \rho)(\sigma \cdot \rho + f\cdot a_{\rho})\Im\big(\iota_{V_{\rho}}(\phi, a_{\rho})\big)+r_{F}(\phi,\sigma,f),
    \end{align*}
    and
    \begin{align*}
        G_{\rho}(\phi,\sigma,f)
        &=
        G_{\rho}(\phi)
        +
        \sin(\phi\cdot \rho) \nabla\sigma \cdot \nabla_{\nabla \phi} \Im(\iota_{V_{\rho}}(\phi, a_{\rho}))
        \\
        &\qquad\qquad+
        \cos(\phi\cdot \rho)(\sigma \cdot \rho + f\cdot a_{\rho})\Re\big(\iota_{V_{\rho}}(\phi, a_{\rho})\big)\\
        &\qquad\qquad+\Im(\iota_{V_{\rho}}(\phi, a_{\rho}))r_{\sin}^{(2)}(\phi,\sigma,f)+r_{\Im}^{(1)}(\phi,\sigma)r_{\sin}^{(1)}(\phi,\sigma,f)+r_{\Im}^{(2)}(\phi,\sigma)\sin( \phi\cdot \rho )\\
        &=
        G_{\rho}(\phi)
        +
        \sin(\phi\cdot \rho) \nabla\sigma \cdot \nabla_{\nabla \phi} \Im(\iota_{V_{\rho}}(\phi, a_{\rho}))
        \\
        &\qquad\qquad+
        \cos(\phi\cdot \rho)(\sigma \cdot \rho + f\cdot a_{\rho})\Re\big(\iota_{V_{\rho}}(\phi, a_{\rho})\big)+r_{G}(\phi,\sigma,f),
    \end{align*}
    where
    \begin{align*}
        &|r_{F}(\phi,\sigma,f)|, |r_{G}(\phi,\sigma,f)|
        \\
        &\quad\leq
        \frac12|\sigma\cdot \rho + a_{\rho}\cdot f|^2
        + 2 (1+\Cr{lnIratio_deriv}c_{\beta})|\sigma\cdot \rho + a_{\rho}\cdot f| |V_{\rho}|^{1/2}\norm{(\nabla \sigma)\restrict_{V_{\rho}}}_2
        \\
        &\qquad +2 (2+\Cr{lnIratio_deriv}c_{\beta})^2|V_{\rho}|\norm{(\nabla \sigma)\restrict_{V_{\rho}}}_2^2
        \\
        &\quad\le |\sigma\cdot \rho + a_{\rho}\cdot f|^2
        +4 (2+\Cr{lnIratio_deriv}c_{\beta})^2|V_{\rho}|\norm{(\nabla \sigma)\restrict_{V_{\rho}}}_2^2
    \end{align*}
    where we used $2ab\le\frac{a^2}{2}+2b^2$ in the last estimate.
    
    For readability, introduce
    \begin{equation*}
        \Tilde{r} = r_{F} + r_G,\qquad H = F_{\rho} - G_{\rho},
    \end{equation*}
    and the first order approximation
    \begin{align*}
        W_{\rho}(\phi,\sigma,f)
        &=
        \cos(\phi\cdot \rho) \nabla\sigma \cdot \nabla_{\nabla \phi} \Re(\iota_{V_{\rho}}(\phi, a_{\rho}))
        +
        \sin(\phi\cdot \rho) \nabla\sigma \cdot \nabla_{\nabla \phi} \Im(\iota_{V_{\rho}}(\phi, a_{\rho}))
        \\
        &\qquad+
        \cos(\phi\cdot \rho)(\sigma \cdot \rho + f\cdot a_{\rho})\Re\big(\iota_{V_{\rho}}(\phi, a_{\rho})\big)
        \\
        &\qquad-
        \sin(\phi\cdot \rho)(\sigma \cdot \rho + f\cdot a_{\rho})\Im\big(\iota_{V_{\rho}}(\phi, a_{\rho})\big).
    \end{align*}Now, for \(|x|\leq 2^{-1}\),
    \begin{equation}
        \big|\ln(1+x)- x\big|\leq x^2,
    \end{equation}so, as \(|z_{\rho}|\leq 8^{-1}\), and \(|H|\leq 2\),
    \begin{equation*}
        \Big|\ln\Big(\frac{1+z_{\rho}H(\phi,\sigma, f)}{1+z_{\rho}H(\phi)}\Big)
        -
        \frac{z_{\rho} (H(\phi,\sigma, f)- H(\phi))}{1+z_{\rho}H(\phi)}\Big|
        \leq
        2z_{\rho}^2|H(\phi,\sigma, f)- H(\phi)|^2.
    \end{equation*}But, by the previous expansions,
    \begin{equation*}
        H(\phi,\sigma, f)- H(\phi)
        =
        W_{\rho}(\phi,\sigma,f) + \Tilde{r}(\phi,\sigma,f).
    \end{equation*}
    Moreover, expanding only to zeroth order, one has
    \begin{equation*}
        |H(\phi,\sigma, f)- H(\phi)|\leq 2|\sigma\cdot \rho + f\cdot a_{\rho}| + 4 (1+\Cr{lnIratio_deriv}c_{\beta})|V_{\rho}|^{1/2} \norm{(\nabla \sigma)\restrict_{V_{\rho}}}_{2}.
    \end{equation*}So, gathering all the estimates we conclude that
    \begin{align*}
        &\Big|\ln\Big(\frac{1+z_{\rho}H(\phi,\sigma, f)}{1+z_{\rho}H(\phi)}\Big)
        -
        \frac{z_{\rho} W_{\rho}(\phi,\sigma,f)}{1+z_{\rho}H(\phi)}\Big|
        \\
        &\quad\leq
        2z_{\rho}^2\Big(2|\sigma\cdot \rho + f\cdot a_{\rho}| + 4(1+\Cr{lnIratio_deriv}c_{\beta}) |V_{\rho}|^{1/2} \norm{(\nabla \sigma)\restrict_{V_{\rho}}}_{2}\Big)^2+\frac{|z_\rho\Tilde{r}(\phi,\sigma,f)|}{1+z_\rho H(\phi)}
        \\
        &\quad\le 16z_\rho^2|\sigma\cdot \rho + f\cdot a_{\rho}|^2+64(1+\Cr{lnIratio_deriv}c_{\beta})^2z_\rho^2|V_\rho|\norm{(\nabla \sigma)\restrict_{V_{\rho}}}_{2}^2
        \\
        &\qquad\qquad+4|z_\rho||\sigma\cdot \rho + a_{\rho}\cdot f|^2
        +16(2+\Cr{lnIratio_deriv}c_{\beta})^2|z_\rho| |V_{\rho}|\norm{(\nabla \sigma)\restrict_{V_{\rho}}}_2^2
    \end{align*}
    and the claimed estimate follows because $|z_\rho|\le8^{-1}$.
\end{proof}

For \(\Ncal\) as in Theorem~\ref{thm:convex_combination_ensembles}, let
\begin{equation}
    S_{\Ncal} = \sum_{\rho\in \Ncal} S_{\rho},\qquad r_{\Ncal} = \sum_{\rho\in \Ncal} r_{\rho},
\end{equation}where \(S_{\rho}, r_{\rho}\) are given in Claim~\ref{claim:Taylor_FG}.

From Claim~\ref{claim:Taylor_FG}, we easily get the next claim, which bounds the ``error term'' \(r_{\Ncal}\) for the relevant choice of \(\sigma\) as a function of \(f\).

\begin{claim}
    \label{claim:Taylor_prod_FG}
   If $\Cr{quantsmallgamma}>0$ is small enough, the following holds. If \(\Ncal\) is as in Theorem~\ref{thm:convex_combination_ensembles}, and \(\sigma = e^{\Cr{energy}\gamma_{\beta}/4} \Delta_{\Lambda}^{-1} f\), then,
    \begin{equation*}
        \prod_{\rho\in \Ncal} \frac{1+ z_{\rho}F_{\rho}(\phi,\sigma,f) - z_{\rho}G_{\rho}(\phi,\sigma,f)}{1+ z_{\rho}F_{\rho}(\phi) - z_{\rho}G_{\rho}(\phi)}
        =
        \exp(S_{\Ncal}(\phi,\sigma, f) + r_{\Ncal}(\phi,\sigma, f)),
    \end{equation*}and \(r_{\Ncal}\) satisfies
    \begin{equation*}
        \big|r_{\Ncal}(\phi,\sigma, f)\big|
        \leq
        2e^{-\Cr{energy}\gamma_\beta/2} \norm{\nabla \sigma}_2^2,
    \end{equation*}uniformly over \(\phi\).
\end{claim}
\begin{proof}
    The identity follows from the definition of \(S_{\Ncal}, r_{\Ncal}\). To get the upper bound on the error term, we will use the next two observations.
    \begin{enumerate}
        \item First, note that by Lemma~\ref{lem:spin_waves}, and Cauchy-Schwarz
        \begin{equation}
        \label{eq:a_rho_scalar_f}
        \begin{aligned}
            |a_{\rho}\cdot f|
            =
            e^{-\Cr{energy}\gamma_{\beta}/4} |\nabla a_{\rho} \cdot \nabla \sigma|
            &\leq 
            e^{-\Cr{energy}\gamma_{\beta}/4}\gamma_{\beta} (\Cr{norm_gradSpinWave}(|\supp \rho|+A_{\Lambda}(\rho)))^{1/2}\norm{(\nabla \sigma)\restrict_{V_\rho}}_{2}
            \\
            &\leq
            (|\supp \rho|+A_{\Lambda}(\rho))^{1/2}\norm{(\nabla \sigma)\restrict_{V_\rho}}_{2}
        \end{aligned}
        \end{equation}where the last inequality holds for \(\gamma_\beta\) large enough (i.e.~$\Cr{quantsmallgamma}$ small enough).
        \item Second, we claim that
        \begin{equation}
        \label{eq:rho_scalar_sigma}
            |\rho \cdot \sigma| 
            \leq
            \norm{\rho}_1 |\DensSquareEnv_{\Lambda}^+(\rho)|^{1/2} \norm{(\nabla\sigma)\restrict_{\DensSquareEnv^+_{\Lambda}(\rho)}}_{2}.
        \end{equation}Indeed, proceeding as in~\cite[Section 4.2.2]{G23}, for each \(\rho\in \Ncal\) with \(\rho\cdot 1\neq 0\), let \(w_*(\rho)\) with \(\dist(w_*(\rho),\Lambda) = 1\) be a vertex such that \(\dist(w_{*}(\rho), \rho) = \dd_{\Lambda}(\rho)\), so that \(w_{*}(\rho) \in \DensSquareEnv_{\Lambda}(\rho)\). Then define
        \begin{equation*}
            \rho_* =
            \begin{cases}
                \rho & \text{ if } \rho \cdot 1 = 0,
                \\
                \rho - \rho \cdot 1 \delta_{w_*(\rho)} & \text{ if } \rho \cdot 1 \neq 0,
            \end{cases}
        \end{equation*}where \(\delta_{w}\) is the charge density which is \(1\) on \(w\) and \(0\) everywhere else. \(\rho_*\) is then neutral and, as \(\sigma\) is supported in \(\Lambda\) and \(w_*(\rho)\in \Lambda^c\), \(\rho \cdot \sigma = \rho_*\cdot \sigma\). Proceeding as in~\cite[Section 4.1.1.]{G23} (write \(\rho_*\) as a sum of dipoles, write each dipoles as a telescopic sum of gradients along a fixed path staying in \(D_{\Lambda}^+(\rho)\), and use that there are at most \(\frac{1}{2}\norm{\rho_*}_1\) such paths passing through a given edge), one obtains
        \begin{equation*}
            |\rho_* \cdot \sigma|
            \leq
            \frac{1}{2}\norm{\rho_*}_1|\DensSquareEnv_{\Lambda}^+(\rho)|^{1/2} \norm{(\nabla\sigma)\restrict_{\DensSquareEnv^+_{\Lambda}(\rho)}}_{2}
            \leq
            \norm{\rho}_1 |\DensSquareEnv_{\Lambda}^+(\rho)|^{1/2} \norm{(\nabla\sigma)\restrict_{\DensSquareEnv^+_{\Lambda}(\rho)}}_{2}.
        \end{equation*}
    \end{enumerate}
    Combining \eqref{eq:a_rho_scalar_f} and \eqref{eq:rho_scalar_sigma}, we obtain that
    \[|\sigma\cdot \rho + a_{\rho}\cdot f|^2\le 2\Big((|\supp \rho|+A_{\Lambda}(\rho))+\norm{\rho}_1^2 |\DensSquareEnv_{\Lambda}^+(\rho)| \Big)\norm{(\nabla\sigma)\restrict_{\DensSquareEnv^+_{\Lambda}(\rho)}}_{2}^2\]

    
    We can then plug this observation into the output of Claim~\ref{claim:Taylor_FG}, as well as the facts that \(V_{\rho}\subset\DensSquareEnv^+_{\Lambda}(\rho) \), that $|\supp\rho|\le \|\rho\|_{1}$, and that \((2+\Cr{lnIratio_deriv}c_{\beta})^2\leq 8+2\Cr{lnIratio_deriv}^2c_{\beta}^2\) to obtain
    \begin{align*}
        \big|r_{\Ncal}(\phi,\sigma, f)\big|
        &\leq
        \sum_{\rho\in \Ncal} |z_{\rho}|\norm{(\nabla\sigma)\restrict_{\DensSquareEnv^+_{\Lambda}(\rho)}}_{2}^2\Big(12(|\supp \rho|+A_{\Lambda}(\rho))
        \\
        &\qquad\qquad+12\norm{\rho}_1^2 |\DensSquareEnv_{\Lambda}^+(\rho)| +(192+48\Cr{lnIratio_deriv}^2c_{\beta}^2))|\DensSquareEnv_{\Lambda}^+(\rho)|\Big)
        \\
        &\leq
        12\sum_{\rho\in \Ncal} |z_{\rho}|\Big(2\norm{\rho}_1^4 +A_\Lambda(\rho)+(18+4\Cr{lnIratio_deriv}^2c_{\beta}^2)|\DensSquareEnv^+_{\Lambda}(\rho)|^2\Big) \norm{(\nabla\sigma)\restrict_{\DensSquareEnv^+_{\Lambda}(\rho)}}_{2}^2
    \end{align*}
    Because $|\DensSquareEnv^+_{\Lambda}(\rho)|\le (4\dd_\Lambda(\rho)+1)^2$ and $A_\Lambda(\rho)\ge\log_2(1+\dd_\Lambda(\rho))$, for $\beta$ sufficiently small we have, using~\eqref{eq:def:gamma_beta}, for $\Cr{quantsmallgamma}$ small enough and hence $\gamma_\beta$ large enough, the estimate
    \begin{align*}
        12\Big(2\norm{\rho}_1^4 +A_\Lambda(\rho)+(18+4\Cr{lnIratio_deriv}^2c_{\beta}^2)|\DensSquareEnv^+_{\Lambda}(\rho)|^2\Big)
        &\leq
        12\Big(2\norm{\rho}_1^4 +A_\Lambda(\rho)+22|\DensSquareEnv^+_{\Lambda}(\rho)|^2\Big) 
        \\
        &\le
        e^{\Cr{energy}\gamma_\beta (\|\rho\|_1+A_\Lambda(\rho))/2}.
    \end{align*}
    Recalling now~\eqref{eq:renorm_activity_upper_bound} and Lemma~\ref{lem:control_A_with_Ssep}, we find that
    \begin{align*}
        \big|r_{\Ncal}(\phi,\sigma, f)\big|
        &\leq
        \sum_{\rho\in \Ncal} e^{-\Cr{energy}\gamma_\beta A_\Lambda(\rho)/2}\norm{(\nabla\sigma)\restrict_{\DensSquareEnv^+_{\Lambda}(\rho)}}_{2}^2\\
        &\leq\sum_{\rho\in \Ncal} e^{-\Cr{energy}\gamma_\beta \log_2(1+\dd_\Lambda(\rho))/2}\norm{(\nabla\sigma)\restrict_{\DensSquareEnv^+_{\Lambda}(\rho)}}_{2}^2\\
        &=\sum_{i\sim j}(\sigma(i)-\sigma(j))^2\sum_{\substack{\rho\in \Ncal\\i,j\in \DensSquareEnv^+_{\Lambda}(\rho)}}e^{-\Cr{energy}\gamma_\beta \log_2(1+\dd_\Lambda(\rho))/2}
    \end{align*}
    To bound the latter sum, one can proceed as in \cite[Proof of Claim 3.2]{KP17}: For two charges $\rho\neq\rho'\in\Ncal$, if $\dd_\Lambda(\rho)$ and $\dd_\Lambda(\rho')$ are within a factor of 2 of each other, then $\DensSquareEnv^+_{\Lambda}(\rho)$ and $\DensSquareEnv^+_{\Lambda}(\rho')$ are disjoint. This means that for any fixed $i\sim j$ we have
    \[\sum_{\substack{\rho\in \Ncal\\i,j\in \DensSquareEnv^+_{\Lambda}(\rho)}}e^{-\Cr{energy}\gamma_\beta \log_2(1+\dd_\Lambda(\rho))/2}\le \sum_{k=1}^\infty e^{-\Cr{energy}\gamma_\beta k/2}\]
    and for $\gamma_\beta$ large enough the right-hand side is bounded by $2e^{-\Cr{energy}\gamma_\beta/2}$.
   This yields the desired estimate.
\end{proof}

\begin{claim}
\label{claim:Taylor_lnI}
    One has
    \begin{align*}
        \ln\Bigg(\frac{\calI_{\Lambda}^{\xi}(\phi+\sigma)}{\calI_{\Lambda}^{\xi}(\phi)}\Bigg)
        &=
        T(\phi,\sigma) + R(\phi,\sigma)
        \\
        &:=
        \sum_{i\sim j} \frac{I_{\beta}'(\phi_i-\phi_j)}{I_{\beta}(\phi_i-\phi_j)}(\sigma_i-\sigma_j)
        + R(\phi,\sigma)
    \end{align*}with \(|R(\phi,\sigma)|\leq \frac{c_{\beta}'}{2} \norm{\nabla \sigma}^2_2\).
\end{claim}
\begin{proof}
    Note that (as \(\sigma\) is supported in \(\Lambda\)),
    \begin{equation*}
        \ln\Bigg(\frac{\calI_{\Lambda}^{\xi}(\phi+\sigma)}{\calI_{\Lambda}^{\xi}(\phi)}\Bigg)
        =
        \sum_{i\sim j} \ln\Bigg(\frac{I_{\beta}(x_{ij}+\sigma_{ij})}{I_{\beta}(x_{ij})}\Bigg)
    \end{equation*}where \(x_{ij} \equiv \phi_i-\phi_j\), and the sum is over nearest neighbor pairs inside \(\Lambda\). Now, by Taylor expansion,
    \begin{equation*}
        \ln\Bigg(\frac{I_{\beta}(x+y)}{I_{\beta}(x)}\Bigg) = \frac{I'_{\beta}(x)}{I_{\beta}(x)} y + r(x,y)
    \end{equation*}with \(|r(x,y)|\leq \frac{c_{\beta}'}{2} y^2\) by property~\ref{item:Ibeta:sub_quad} of \(I_{\beta}\). This gives the wanted result.
\end{proof}

\begin{remark}\label{r:techimprovement}
The main difference between our argument and the one in \cite[Section 7]{FS81a} is the treatment of the terms $f\cdot a_{\Ncal,\rho}$ that appear in \eqref{eq:integral_rep_ratio_ensemble_PF} (via the definition of $F_\rho$ and $G_\rho$). The support of $a_{\Ncal,\rho}$ is contained in $\DensSquareEnv_{\Lambda}(\rho)\cap \Lambda$ by Lemma \ref{lem:spin_waves}, and so these terms are zero unless $\supp(f)$ intersects $\DensSquareEnv_{\Lambda}(\rho)\cap \Lambda$. This means that if the support of $f$ has small cardinality, then only a few factors in \eqref{eq:integral_rep_ratio_ensemble_PF} are problematic. In \cite[(7.37)-(7.39)]{FS81a} the authors estimate these roughly using absolute values, which leads to an estimate like ours, but with an additional factor $e^{-C|\supp(f)|}$ on the right-hand side. This loss of a constant factor is stated in \cite[(7.40)]{FS81a}, but missing from the statement of \cite[(7.18)]{FS81a}.

As mentioned in Subsection~\ref{subsec:mainideas}, we cannot afford the loss of this factor $e^{-C|\supp(f)|}$ for our purposes, and we avoid its loss by including the terms $f\cdot a_{\Ncal,\rho}$ into the Taylor expansion. This is conceptually not too different from \cite{FS81a}, but makes the computations more delicate and requires new estimates like \eqref{eq:a_rho_scalar_f}.
\end{remark}

\subsection{Lower bound without symmetrization}
\label{subsec:LB_without_sym}

In this section, we conclude the proof of Theorem~\ref{thm:LB_trigo_poly_non_sym}.
We plug the results of Claims~\ref{claim:Taylor_FG} and~\ref{claim:Taylor_lnI} in~\eqref{eq:integral_rep_ratio_ensemble_PF} to obtain that for \(\epsilon>0\),
\[
    \frac{Z_{\Ncal}^{\zeta}(\epsilon f; \epsilon\sigma) }{Z_{\Ncal}^{\zeta}}
    =
    \int d\nu_{\Lambda}^{\xi}(\phi) e^{\epsilon f\cdot \phi+T(\phi,\epsilon \sigma) +S_{\Ncal}(\phi,\epsilon \sigma, \epsilon f)}e^{\epsilon^2\sigma\cdot f+R(\phi,\epsilon \sigma) + r_{\Ncal}(\phi,\epsilon \sigma, \epsilon f)}
\]
Here by the definitions of $T$ and $S_\Ncal$, the first exponent is linear in $\epsilon$, i.e.
\begin{equation*}
    \epsilon f\cdot \phi+ T(\phi,\epsilon \sigma)+S_{\Ncal}(\phi,\epsilon \sigma, \epsilon f)
    =
    \epsilon( f\cdot \phi+ T(\phi, \sigma)+S_{\Ncal}(\phi, \sigma, f)).
\end{equation*}
The second exponent is quadratic in $\epsilon$, and in fact we have the estimate
\begin{equation}\label{eq:estquadraticqualitative}
    \epsilon^2\sigma\cdot f+R(\phi,\epsilon \sigma) + r_{\Ncal}(\phi,\epsilon \sigma, \epsilon f)= O(\epsilon^2)
\end{equation}
where the implicit constant will in general depend on $f,\sigma,\rho,\Ncal$.

As a consequence
\[
    \frac{Z_{\Ncal}^{\zeta}(\epsilon f; \epsilon\sigma) }{Z_{\Ncal}^{\zeta}}
    =
\int d\nu_{\Lambda}^{\xi}(\phi)e^{\epsilon f\cdot \phi} e^{T(\phi,\epsilon \sigma)} e^{S_{\Ncal}(\phi,\epsilon \sigma, \epsilon f)}e^{O(\epsilon^2)}.
\]

Plugging this expression in~\eqref{eq:multipol_exp:correlation_II}, one gets
\begin{equation*}
    \frac{\int d\mu_{\Lambda}^{\xi} e^{\epsilon f\cdot \phi}\lambda_{\Lambda}(\phi-\zeta)}{\int d\mu_{\Lambda}^{\xi} \lambda_{\Lambda}(\phi-\zeta)}
    =
    e^{O(\epsilon^2)} \sum_{\Ncal\in \Fcal}c_{\Ncal}' \int d\nu_{\Lambda}^{\xi}(\phi)e^{\epsilon f\cdot \phi} e^{T(\phi,\epsilon \sigma)} e^{S_{\Ncal}(\phi,\epsilon \sigma, \epsilon f)}.
\end{equation*}
where the implicit constant still depends on $f,\sigma,\Ncal$, but is finite because \(\Fcal\) is finite. 
Expanding both sides to first order in \(\epsilon\), subtracting the order \(0\) term (\(1\) on both sides), dividing by \(\epsilon\) and letting \(\epsilon\to 0\), we obtain
\begin{equation}
\label{eq:mean_value_linear_obs}
    \frac{\int d\mu_{\Lambda}^{\xi} f\cdot \phi \lambda_{\Lambda}(\phi-\zeta)}{\int d\mu_{\Lambda}^{\xi} \lambda_{\Lambda}(\phi-\zeta)}
    =
    \sum_{\Ncal\in \Fcal}c_{\Ncal}' \int d\nu_{\Lambda}^{\xi}(\phi) \big(f\cdot \phi+ T(\phi, \sigma)+S_{\Ncal}(\phi, \sigma, f)\big).
\end{equation}This is valid for any \(f,\sigma\). This is the analogue of the ``modular invariance identity'' \cite[Proposition 4.1]{GS24}. 

In order to proceed further, we first choose $\Cr{lowerboundmainthm}$. In fact we will set $\Cr{lowerboundmainthm}=\frac{\Cr{energy}}{4}$. Next we make a suitable choice of \(\sigma\): the one used in Claim~\ref{claim:Taylor_prod_FG}. Let \(\sigma = e^{\Cr{energy}\gamma_{\beta}/4}\Delta_{\Lambda}^{-1} f=e^{\Cr{lowerboundmainthm}\gamma_{\beta}}\Delta_{\Lambda}^{-1} f\) , so that \(f = e^{-\Cr{lowerboundmainthm}\gamma_{\beta}} \Delta_{\Lambda} \sigma\), and \(f\cdot \sigma = e^{-\Cr{lowerboundmainthm}\gamma_{\beta}} \norm{\nabla \sigma}_{2}^2 = e^{\Cr{lowerboundmainthm}\gamma_{\beta}} f\cdot \Delta_{\Lambda}^{-1} f\). 
We can now give a more precise version of \eqref{eq:estquadraticqualitative}. Namely as a consequence of as a combination of Claims~\ref{claim:Taylor_prod_FG} and~\ref{claim:Taylor_lnI} we have for $\beta$ small enough
\begin{equation}\label{eq:estquadraticquantitative}
    \sigma \cdot f + R(\phi,\epsilon \sigma) + r_{\Ncal}(\phi,\epsilon \sigma, \epsilon f)
    \geq
    \big(e^{-\Cr{lowerboundmainthm}\gamma_{\beta}} - \tfrac{1}{2}c_{\beta}' -2e^{-2\Cr{lowerboundmainthm}\gamma_{\beta}}\big)\norm{\nabla \sigma}_2^2\ge\frac12e^{-\Cr{lowerboundmainthm}\gamma_{\beta}}\norm{\nabla \sigma}_2^2
\end{equation}
where the last estimate holds if $\gamma_\beta$ is large enough and 
$c_\beta'e^{\Cr{lowerboundmainthm}\gamma_{\beta}}\le\frac18$, and this in turn follows from \eqref{eq:def:gamma_beta} if we pick $\Cr{quantsmallgamma}$ small enough.

Then, for \(\epsilon >0\) and \(\beta\) sufficiently small, we obtain from~\eqref{eq:multipol_exp:correlation_II} and~\eqref{eq:estquadraticquantitative} that 
\begin{multline}\label{eq:lowerboundshorthand}
    \frac{\int d\mu_{\Lambda}^{\xi}(\phi) e^{\epsilon f\cdot \phi} \lambda_{\Lambda}(\phi-\zeta)}{\int d\mu_{\Lambda}^{\xi} (\varphi)\lambda_{\Lambda}(\varphi-\zeta)}\\
    \geq
    \exp\Big( \tfrac12\epsilon^2e^{\Cr{lowerboundmainthm}\gamma_{\beta}} f\cdot \Delta_{\Lambda}^{-1} f\Big) \sum_{\Ncal\in \Fcal}c_{\Ncal}' \int d\nu_{\Lambda}^{\xi}(\phi) e^{\epsilon (f\cdot \phi+T(\phi, \sigma) + S_{\Ncal}(\phi, \sigma, f))}.
\end{multline}

We introduce now the shorthands \(g(\Ncal,\phi) \equiv f\cdot \phi+T(\phi, \sigma) + S_{\Ncal}(\phi, \sigma, f)\), \(t\equiv e^{\Cr{lowerboundmainthm}\gamma_{\beta}}\), and
\begin{equation*}
    d\tilde{\mu}(\phi) \equiv \frac{1}{\int d\mu_{\Lambda}^{\xi}(\varphi) \lambda_{\Lambda}(\varphi-\zeta)} \lambda_{\Lambda}(\phi-\zeta) d\mu_{\Lambda}(\phi),
    \quad
    d\tilde{\nu}(\Ncal,\phi) \equiv c_{\Ncal}'d\nu_{\Lambda}^{\xi}(\phi),
\end{equation*}
where we see \(\tilde{\nu}\) as a probability measure on \(\Fcal\times \R^{\Lambda}\). Then we have
\begin{align}
    \int d\tilde{\mu}(\phi) f\cdot \phi
    &=
    \int d\tilde{\nu}(\Ncal,\phi) g(\Ncal,\phi)\label{eq:eqfirstorder},
    \\
    \int d\tilde{\mu}(\phi) e^{\epsilon f\cdot \phi}
    &\geq
    e^{ \epsilon^2 t f\cdot \Delta_{\Lambda}^{-1} f/2} \int d\tilde{\nu}(\Ncal,\phi) e^{\epsilon g(\Ncal,\phi)}\label{eq:ineqexpmoment}
\end{align}
where the first line is~\eqref{eq:mean_value_linear_obs}, and the second~\eqref{eq:lowerboundshorthand}. When we expand both sides of \eqref{eq:ineqexpmoment} to second order in \(\epsilon\), the zeroth order terms are both equal to 1, and the first order terms of both sides agree by \eqref{eq:eqfirstorder}. Hence, subtracting them, dividing by \(\epsilon^2/2\) and letting \(\epsilon\to 0\), we get
\begin{align*}
    \int d\tilde{\mu}(\phi) (f\cdot \phi)^2
    &\geq
    t f\cdot \Delta_{\Lambda}^{-1} f +  \int d\tilde{\nu}(\Ncal,\phi) g(\Ncal,\phi)^2
    \\
    &\geq
    t f\cdot \Delta_{\Lambda}^{-1} f +  \Big(\int d\tilde{\nu}(\Ncal,\phi) g(\Ncal,\phi) \Big)^2\\
    &=
    t f\cdot \Delta_{\Lambda}^{-1} f +  \Big(\int d\tilde{\mu}(\phi) f\cdot \phi\Big)^2
\end{align*}where we used the Cauchy-Schwarz inequality in the second line, and~\eqref{eq:mean_value_linear_obs} in the last line. Re-ordering the terms and replacing the shorthands by their definition, we get Theorem~\ref{thm:LB_trigo_poly_non_sym}.

\subsection{Lower bound via symmetrization}
\label{subsec:LB_sym}

We now prove Theorem~\ref{thm:LB_trigo_poly_sym}. Recall the definition of the reflection symmetry from \eqref{eq:def_reflectionsymmetry}.

Under the assumptions of Theorem~\ref{thm:LB_trigo_poly_sym} we know that $f\cdot\phi=-f\cdot\phi^{\refl}$, $I(\phi)=I(-\phi^{\refl})$ and (because trigonometric polynomials are even functions) also
\[\lambda_{\Lambda}(\phi-\zeta)=\lambda_{\Lambda}\big(-(\phi-\zeta)^{\refl}\big)=\lambda_{\Lambda}\big(-\phi^{\refl}-\zeta\big).\]
So, if we make the change of variables $\varphi=-\phi^{\refl}$, we obtain that
\begin{equation*}
    \int d\phi \calI_{\Lambda}^{\xi}(\phi) e^{f\cdot \phi} \lambda_{\Lambda}(\phi-\zeta)
    =
    \int d\phi \calI_{\Lambda}^{\xi}(\phi) e^{f\cdot \phi^{\refl}} \lambda_{\Lambda}(-\phi^{\refl}-\zeta)
    =
    \int d\varphi \calI_{\Lambda}^{\xi}(\varphi) e^{-f\cdot \varphi} \lambda_{\Lambda}(\varphi-\zeta)
\end{equation*}
and hence also
\begin{equation*}
    \int d\phi \calI_{\Lambda}^{\xi}(\phi) e^{f\cdot \phi} \lambda_{\Lambda}(\phi-\zeta)=\frac12\Bigg(\int d\phi \calI_{\Lambda}^{\xi}(\phi) e^{f\cdot \phi} \lambda_{\Lambda}(\phi-\zeta)+\int d\phi \calI_{\Lambda}^{\xi}(\phi) e^{-f\cdot \phi} \lambda_{\Lambda}(\phi-\zeta)\Bigg)
\end{equation*}

We can then use the above and the exact same procedure leading to~\eqref{eq:multipol_exp:correlation_I} to obtain that for \(\sigma\in \R^{\Lambda}\) with \(\sigma_i = \sigma_{-i}\),
\begin{equation}\label{eq:symm_expansion}
    \int d\phi \calI_{\Lambda}^{\xi}(\phi) e^{f\cdot \phi} \lambda_{\Lambda}(\phi)
    =
    \frac{1}{2}\sum_{\Ncal\in \Fcal}c_{\Ncal} \big(Z^\zeta_{\Ncal}(f;\sigma) +Z^\zeta_{\Ncal}(-f;-\sigma)\big),
\end{equation}where, as before,
\begin{equation*}
    Z^\zeta_{\Ncal}(\pm f;\pm\sigma)
    =
    e^{\sigma\cdot f}\int d\phi \calI_{\Lambda}^{\xi}(\phi \pm\sigma) e^{\pm f\cdot \phi} \prod_{\rho\in \Ncal}\big(1 + 2K(\rho)\cos(\phi\cdot \rho \pm \sigma \cdot \rho)\big).
\end{equation*}

Next, we can renormalize the activities exactly as in Section~\ref{subsec:activity_renorm} to obtain that \(Z_{\Ncal}(\pm f;\pm\sigma)\) is equal to
\begin{equation*}
    e^{\sigma\cdot f}\int d\phi \calI_{\Lambda}^{\xi}(\phi\pm \sigma) e^{\pm f\cdot \phi} \prod_{\rho\in \Ncal} \Big(1
    +
    z_{\rho} F_{\rho}(\phi,\pm \sigma,\pm f)
    \\
    -
    z_{\rho} G_{\rho}(\phi,\pm \sigma,\pm f)\Big),
\end{equation*}with \(F_{\rho}, G_{\rho}\) as in Section~\ref{subsec:Taylor}, and the same for \(Z_{\Ncal}(f;\sigma)\). We can use the same choice of \(\sigma\) as in the previous section:
\begin{equation*}
    \sigma = e^{\Cr{lowerboundmainthm}\gamma_{\beta}}\Delta_{\Lambda}^{-1} f.
\end{equation*}

Now, by Claim~\ref{claim:Taylor_prod_FG}, one has
\begin{multline*}
    \prod_{\rho\in \Ncal} \frac{1+ z_{\rho}F_{\rho}(\phi,\pm\sigma,\pm f) - z_{\rho}G_{\rho}(\phi,\pm\sigma,\pm f)}{1+ z_{\rho}F_{\rho}(\phi) - z_{\rho}G_{\rho}(\phi)}\\
    \geq
    \exp\big(S_{\Ncal}(\phi,\pm\sigma,\pm f) - 2e^{-\Cr{energy}\gamma_{\beta}/2} \norm{\nabla\sigma}_2^2 \big)
\end{multline*}
with \(S_{\Ncal}\) defined as in Section~\ref{subsec:Taylor}.
Also, by Claim~\ref{claim:Taylor_lnI}, one has
\begin{equation*}
    \frac{\calI_{\Lambda}^{\xi}(\phi \pm \sigma)}{\calI_{\Lambda}^{\xi}(\phi)}
    \geq
    \exp\big(T(\phi,\pm\sigma) - \tfrac{c_{\beta}'}{2}\norm{\nabla \sigma}_2^2\big)
\end{equation*}with \(T(\phi,\sigma)\) defined as in Claim~\ref{claim:Taylor_lnI}. 
We also know that \(f\cdot \sigma = e^{-\Cr{energy}\gamma_{\beta}/4} \norm{\nabla \sigma}_{2}^2\). Using these estimates, we find
\begin{align*}
    Z^\zeta_{\Ncal}(\pm f;\pm \sigma)&\geq \exp\Big(\big(e^{-\Cr{lowerboundmainthm}\gamma_{\beta}} - \tfrac{1}{2}c_{\beta}' -2e^{-2\Cr{lowerboundmainthm}\gamma_{\beta}}\big)\norm{\nabla \sigma}_2^2\Big)\\
    &\qquad\qquad\times\int d\phi \calI_{\Lambda}^{\xi}(\phi)
   e^{\pm f\cdot \phi + T(\phi,\pm\sigma)+S_{\Ncal}(\phi,\pm\sigma,\pm f)}\prod_{\rho\in \Ncal} \big(1
    +
    z_{\rho} F_{\rho}(\phi)
    -
    z_{\rho} G_{\rho}(\phi)\big)\\
    &\ge\exp\Big(\tfrac12e^{-\Cr{lowerboundmainthm}\gamma_{\beta}}\norm{\nabla \sigma}_2^2\Big)\\
    &\qquad\qquad\times\int d\phi \calI_{\Lambda}^{\xi}(\phi)
   e^{\pm f\cdot \phi + T(\phi,\pm\sigma)+S_{\Ncal}(\phi,\pm\sigma,\pm f)}\prod_{\rho\in \Ncal} \big(1
    +
    z_{\rho} F_{\rho}(\phi)
    -
    z_{\rho} G_{\rho}(\phi)\big)
\end{align*}
where the last estimate is as in \eqref{eq:estquadraticquantitative}.

When we add the two estimates for $Z_{\Ncal}(\pm f;\pm \sigma)$, we can use that $T$ and $S_\Ncal$ are linear in $(f,\sigma)$ and hence \(T(\phi,-\sigma) = -T(\phi,\sigma)\), \(S_{\Ncal}(\phi,-\sigma,-f) = -S_{\Ncal}(\phi,\sigma,f)\). This means that we can use the simple \(e^{x}+e^{-x} \geq 2\) to obtain a lower bound. In detail, we obtain
\begin{align*}
    &Z^\zeta_{\Ncal}(f;\sigma) +Z^\zeta_{\Ncal}(-f;-\sigma)\\
    &\geq
    \exp\Big(\tfrac12e^{-\Cr{lowerboundmainthm}\gamma_{\beta}}\norm{\nabla \sigma}_2^2\Big)\int d\phi \calI_{\Lambda}^{\xi}(\phi)
    \\
    &\qquad\qquad\times
    \Big(e^{f\cdot \phi + T(\phi,\sigma)+S_{\Ncal}(\phi,\sigma,f)} + e^{-f\cdot \phi - T(\phi,\sigma) - S_{\Ncal}(\phi,\sigma,f)}\Big)\prod_{\rho\in \Ncal} \big(1
    +
    z_{\rho} F_{\rho}(\phi)
    -
    z_{\rho} G_{\rho}(\phi)\big)
    \\
    &\geq2\exp\Big(\tfrac12e^{-\Cr{lowerboundmainthm}\gamma_{\beta}4}\norm{\nabla \sigma}_2^2\Big)\int d\phi \calI_{\Lambda}^{\xi}(\phi)\prod_{\rho\in \Ncal} \big(1
    +
    z_{\rho} F_{\rho}(\phi)
    -
    z_{\rho} G_{\rho}(\phi)\big)\\
    &=2\exp\Big(\tfrac12e^{-\Cr{lowerboundmainthm}\gamma_{\beta}}\norm{\nabla \sigma}_2^2\Big)Z_{\Ncal}
\end{align*}

Using \eqref{eq:symm_expansion} and its version for $f=\sigma=0$, namely
\begin{equation*}
    \int d\phi \calI_{\Lambda}^{\xi}(\phi) \lambda_{\Lambda}(\phi) = \sum_{\Ncal\in \Fcal}c_{\Ncal} Z^\zeta_{\Ncal},
\end{equation*} we conclude that
\begin{equation*}
    \frac{\int d\phi \calI_{\Lambda}^{\xi}(\phi) e^{f\cdot \phi} \lambda_{\Lambda}(\phi)}{\int d\phi \calI_{\Lambda}^{\xi}(\phi) \lambda_{\Lambda}(\phi)}
    \geq
    \exp\Big(\tfrac12e^{-\Cr{lowerboundmainthm}\gamma_{\beta}}\norm{\nabla \sigma}_2^2\Big)
    =
    \exp\Big(\tfrac12e^{\Cr{lowerboundmainthm}\gamma_{\beta}}f\cdot\Delta_\Lambda^{-1}f\Big).
\end{equation*} This is what we needed to prove.

\section{\texorpdfstring{Proof of results for \(p\)-SOS model}{Proof of results for p-SOS model}}
\label{sec:trigo_poly_to_sum}

In this section, we prove that the SOS (or more generally, the \(p\)-SOS) measure can be put in the framework of the previous section, thus deducing Theorem~\ref{thm:main:p_SOS} from Theorems~\ref{thm:LB_trigo_poly_non_sym} and~\ref{thm:LB_trigo_poly_sym}.

\subsection{Analytic extension}

We start by showing that the \(p\)-SOS models can be put in the framework of Section~\ref{subsec:models}.
\begin{lemma}\label{p:heightfunctextension}
The following integer-valued height functions can be rewritten in the form required in Section~\ref{subsec:models}:
\begin{itemize}
    \item The integer-valued GFF, where
    \[I_\beta(n)=\exp\left(-\beta|n|^2\right)\quad\text{for }n\in\Z\]
    \item The integer-valued solid-on-solid model, where
    \[I_\beta(n)=\exp\left(-\beta|n|\right)\quad\text{for }n\in\Z\]
    \item More generally, the $p$-SOS-model for $p\in(0,2]$, given by
    \[I_\beta(n)=\exp\left(-\beta|n|^p\right)\quad\text{for }n\in\Z\]
    \item The height function dual to the XY-model at inverse temperature $\beta$, given by
    \[I_\beta(n)=\mathsf{I}_{n}\left(\frac{1}{\beta}\right)\quad\text{for }n\in\Z\]
    where 
    \[\mathsf{I}_\nu(z)=\frac{1}{\pi}\int_0^\pi e^{z\cos(t)}\cos(\nu t) dt\] denotes the modified Bessel function of the first kind of order $\nu$. See \cite[Section 6]{FS81a} and \cite{vEL23b} for an explanation of how this height function is related to the XY-model.
\end{itemize}

\end{lemma}
This result is trivial in the case of the GFF, and the case of the SOS model and the dual of the XY model are established in \cite[Appendix B and C]{FS81a} (apart from the issue mentioned in Remark~\ref{r:difftoFS}). We provide a detailed proof for the case of the $p$-SOS-model in Appendix~\ref{app:analytic_ext}. The proof there also shows that one can choose the parameters in Subsection~\ref{subsec:models} as we indicated in Remark~\ref{r:choiceparam}, and that $I_\beta$ has stretched-exponential tails at speed $p$ in the sense that
\begin{equation}\label{eq:tailsIbeta}
|I_\beta(x)|\le Ce^{-\beta|x|^p+C\beta^{p/(p+2)}}\le C'_\beta e^{-\beta|x|^p}
\end{equation}
where $C'_\beta$ is some $\beta$-dependent constant.

\subsection{Reduction to trigonometric polynomials}
Now that we know that the setting of Subsection~\ref{subsec:models} applies in the case of the $p$-SOS model, the only remaining difficulty in proving Theorem \ref{thm:main:p_SOS} is to ensure that we can pass to suitable limits in Theorem~\ref{thm:LB_trigo_poly_non_sym} and Theorem~\ref{thm:LB_trigo_poly_sym}.

Recall that, as tempered distributions,
\begin{equation*}
    \comb(x) = \sum_{n\in \Z}\delta(x-n) = \sum_{n\in \Z}e^{2\pi \rmi x n}.
\end{equation*}
Define the trigonometric polynomials
\begin{gather*}
    \lambda^{\leq N}(x) = \sum_{q=-N}^N e^{2\pi \rmi q x} = 1 + 2 \sum_{n=1}^{N}\cos(2\pi q x),
    \\
    \lambda_{\Lambda}^{\leq N}(x) = \prod_{i\in \Lambda} \lambda^{\leq N}(x_i).
\end{gather*}

We then have a standard result on convergence of trigonometric polynomials.
\begin{lemma}
    \label{lem:trigo_poly_to_Dirac_comb}
    For any \(f:\R^{\Lambda}\to \R\) in the Schwartz space,
    \begin{equation*}
        \int_{\R^{\Lambda}} dx f(x) \lambda_{\Lambda}^{\leq N}(x) \xrightarrow{N\to\infty}
        \sum_{n\in \Z^{\Lambda}} f(n).
    \end{equation*}
\end{lemma}
\begin{proof}
This is equivalent to the fact that $\lambda^{\leq N}$ converge to $\comb$ in the space of tempered distributions on $\R^\Lambda$. For the latter fact, see e.g.~\cite[Chapter 9.2, Exercise 22]{F99}.
\end{proof}

\subsection{\texorpdfstring{Lower bound for $p$-SOS models}{Lower bound for p-SOS models}}

We now combine Lemmas~\ref{lem:trigo_poly_to_Dirac_comb} and Theorems~\ref{thm:LB_trigo_poly_non_sym} and~\ref{thm:LB_trigo_poly_sym} to 
establish our main result for $p$-SOS models, Theorem~\ref{thm:main:p_SOS}.

First, for \(p\in (1,2]\), \(I_{\beta}\) has super-exponential tails. Moreover, the analytic extension constructed in Lemma~\ref{p:heightfunctextension} is a Schwartz function, and the product \(\calI_{\Lambda}^{\xi}(\phi)e^{f\cdot \phi}\) is also a Schwartz function. We can therefore use Lemma~\ref{lem:trigo_poly_to_Dirac_comb} to go from Theorems~\ref{thm:LB_trigo_poly_non_sym} and~\ref{thm:LB_trigo_poly_sym} to Theorem~\ref{thm:main:p_SOS}. The case \(p=1\) is handled in the same fashion after adding the restriction on \(\norm{f}_{\infty}\).

If \(p\in (0,1)\), the above argument no longer works, as the stretched-exponential tails of $I_\beta$ can no longer compensate the exponential growth of $e^{f\cdot \phi}$, and so the exponential moments are no longer defined. So we need to regularize the model first.

Given a parameter $\varepsilon\ll1$, let \(I_{\beta}^{\varepsilon}:\Z\to \Z\) be given by
\begin{equation*}
    I_{\beta}^\varepsilon(x) =
    e^{-\varepsilon x^2}I_{\beta}(x).
\end{equation*}Define \(\calI_{\Lambda}^{\varepsilon,\xi}\) as \(\calI_{\Lambda}^{\xi}\) with \(I_{\beta}^\varepsilon\) replacing \(I_{\beta}\).

Suppose that \(I_{\beta}\) has at least stretched-exponential tails. Then for any \(a\in (\Z_{\geq 0})^{\Lambda}\),
\begin{equation*}
    \sum_{\phi\in \zeta + \Z^{\Lambda}} \calI_{\Lambda}^{\varepsilon,\xi}(\phi) \phi^{a}
    \xrightarrow{\varepsilon\to0}
    \sum_{\phi\in \zeta + \Z^{\Lambda}}  \calI_{\Lambda}^{\xi}(\phi) \phi^{a}.
\end{equation*}
pointwise

In the case of the $p$-SOS model, the analytic extension $I_\beta$ as constructed in Lemma~\ref{p:heightfunctextension} is a Schwartz function with stretched-exponential tails. But \(I_{\beta}^{\varepsilon}\) then has an extension with Gaussian tails (the straightforward one). In particular, approximating it as before, one obtains from Theorem~\ref{thm:LB_trigo_poly_non_sym} that for \(\varepsilon>0\) fixed small enough,
\begin{equation*}
    \int d\mu_{\Lambda;\beta}^{\varepsilon,\xi,\zeta, \IV}(\phi) (f\cdot \phi)^2 - \Big(\int d\mu_{\Lambda;\beta}^{\varepsilon,\xi,\zeta,\IV}(\phi) f\cdot \phi\Big)^2
    \geq e^{\Cr{lowerboundmainthm}\gamma_{\beta}} f\cdot \Delta_{\Lambda}^{-1} f
\end{equation*}where \(\mu_{\Lambda;\beta}^{\varepsilon,\xi,\zeta,\IV}\) is defined as \(\mu_{\Lambda;\beta}^{\xi,\zeta,\IV}\) using \(\calI_{\Lambda}^{\varepsilon,\xi}\) in place of \(\calI_{\Lambda}^{\xi}\). Letting \(\varepsilon\to 0\) and using dominated convergence then gives Theorem~\ref{thm:main:p_SOS} also for $0<p<1$.

\appendix

\section{Spin wave construction}
\label{sec:spin_waves}

In this section we present the construction of the spin waves that are used in the renormalization of the multipole activities. It is basically the same construction as the ones of~\cite[Section 4]{FS81a}, and~\cite{KP17}, so we only sketch the arguments and refer to the relevant parts of~\cite{KP17}.
Recall \(g\) from Subsection~\ref{subsec:models}.

Introduce the following notation: for a function \(f\) from the edges of \(\Z^2\) to \(\R\), define
\begin{equation*}
    \supp f = \{i\in \Z^2:\ \exists j\sim i, f(\{i,j\})\neq 0 \}.
\end{equation*}
For a charge ensemble \(\Ncal\), and \(\rho\in \Ncal\), define
\begin{equation*}
    \Ncal_{\lesssim }(\rho) = \{\rho'\in \Ncal:\ \dd(\rho')\leq 2\dd(\rho)\}.
\end{equation*}Also recall
\begin{equation*}
    E_{\beta}(\rho, a) = \rho \cdot a - (2\Cr{lnIratio_deriv}+1)c_{\beta}\sum_{i\sim j} g(a_i-a_j).
\end{equation*}

\begin{lemma}
    \label{lem:spin_wave_level_0}
    \(\Lambda\subset \Z^2\). Let \(\Ncal\) be a charge ensemble as given by Theorem~\ref{thm:convex_combination_ensembles}. Let \(\rho\in \Ncal\). Then, there is \(b_{\rho,0}:\Z^2 \to \R\) such that
    \begin{enumerate}
        \item \(\supp b_{\rho,0}\subset \supp \rho\);
        \item for all \(i\sim j\), \(|b_{\rho,0}(i)-b_{\rho,0}(j)| \leq 1\);
        \item for every \(\rho'\in \Ncal_{\lesssim}(\rho)\), \(b_{\rho,0}\) is constant on \(\DensSquareEnv^+(\rho')\);
        \item one has \(\rho \cdot b_{\rho,0} \geq \frac{1}{2}\norm{\rho}_1\).
    \end{enumerate}
\end{lemma}
\begin{proof}
    \(\Z^2\) being bipartite, for at least one of the two classes, \(V\), it holds that
    \begin{itemize}
        \item \(\sum_{i\in V} |\rho(i)| \geq \frac{1}{2}\norm{\rho}_1\);
        \item if \(\dd_{\Lambda}(\rho) =1\), \(\DensCenter(\rho)\in V\).
    \end{itemize}
    Take
    \begin{equation*}
        b_{\rho,0}(i)
        =
        \begin{cases}
            1 & \text{ if } i\in V \text{ and } \rho_i >0,
            \\
            -1 & \text{ if } i\in V \text{ and } \rho_i <0,
            \\
            0 & \text{ else}.
        \end{cases}
    \end{equation*}
\end{proof}

The next Lemma is (a mild adaptation of)~\cite[Proposition 4.7]{KP17}, and~\cite[Proposition 26]{W19}.
\begin{lemma}
    \label{lem:spin_wave_square}
    Let \(\Lambda\subset \Z^2\). Let \(\Ncal\) be a charge ensemble as given by Theorem~\ref{thm:convex_combination_ensembles}. Let \(\rho\in \Ncal\). Then, there are \(b_{\rho,k}^s:\Z^2 \to \R\), \(1\leq k\leq \scale_{\Lambda}(\rho)\), \(s\in \SquareCover^{\sep}_k(\rho)\), such that
    \begin{enumerate}
        \item \(\supp b_{\rho,k}^s \subset \{j\in \Lambda:\ \dist(j,s)\leq 2^{k-1}\}\);
        \item \( b_{\rho,k}^s\) is constant on \(\{j\in \Lambda:\ \dist(j,s)\leq 2^{k-3}\vee 1\}\);
        \item for every \(\rho'\in \Ncal_{\lesssim}(\rho)\), \( b_{\rho,k}^s\) is constant on \(\DensSquareEnv^+(\rho')\);
        \item for all \(i\sim j\) with $i,j\not\in\bigcup_{\rho'\in \Ncal_{\lesssim}(\rho)}\DensSquareEnv^+(\rho')$, 
        \begin{equation}\label{eq:spin_wave_square_awayfromclust}
        |b_{\rho,k}^s(i)- b_{\rho,k}^s(j)|\leq \frac{1}{2^k}
        \end{equation}
        while if $i\sim j$, $i\in\bigcup_{\rho'\in \Ncal_{\lesssim}(\rho)}\DensSquareEnv^+(\rho')$, $j\not\in\bigcup_{\rho'\in \Ncal_{\lesssim}(\rho)}\DensSquareEnv^+(\rho')$ then
        \begin{equation}\label{eq:spin_wave_square_atclust}
        |b_{\rho,k}^s(i)- b_{\rho,k}^s(j)|\leq \frac{\diam(E)}{2^k}
        \end{equation}
        where $E$ is the connected cluster in $\bigcup_{\rho'\in \Ncal_{\lesssim}(\rho)}\DensSquareEnv^+(\rho')$ that contains $i$;
        \item we have that
        \begin{equation*}
            \rho \cdot b_{\rho,k}^s \geq \frac{1}{2^9}.
        \end{equation*}
    \end{enumerate}
    In particular, for any \((k,s)\neq(k',s')\), \(\supp \nabla b_{\rho,k}^s \cap \supp b_{\rho,k'}^{s'} = \varnothing\).
\end{lemma}
\begin{proof}
    We mimic the proof in~\cite[Sections 4.3.2, 4.3.3]{KP17}. Distinguish two cases: \(k< 10\) and \(k\geq 10\).
    
    \vspace{0.3cm}

    \noindent\textbf{Case \(k< 10\):}
    
    \vspace{0.3cm}

    Take
    \begin{equation*}
        b_{\rho,k}^{s}(i)
        =
        \begin{cases}
        \frac{\rho \cdot \mathds{1}_s}{|\rho \cdot \mathds{1}_s|} 2^{-k} & \text{ if } \dist(i,s)\leq 1,
        \\
        0 & \text{ else},
        \end{cases}
    \end{equation*}where \(\rho \cdot \mathds{1}_s = \sum_{i\in s} \rho_i\) is the charge of \(\rho\) in \(s\) (which is non-zero by the properties of \(\Ncal\) given in Theorem~\ref{thm:convex_combination_ensembles}, cf. the argument in~\cite[Proof of Proposition 26]{W19}). Properties 1. and 2. are direct from the definition. Property 3. is trivial, as by the same argument as in~\cite[Proof of Proposition 4.7]{KP17} there is no $\rho'\in \Ncal_{\lesssim}(\rho)$ such that $\DensSquareEnv^+(\rho')$ intersects $\supp b_{\rho,k}^s$.
    For the same reason, regarding Property 4. only \eqref{eq:spin_wave_square_awayfromclust} is relevant, and this is again obvious from the definition. Finally, Property 5. follows from $k<10$ and 
    \begin{equation*}
        \rho \cdot b_{\rho,k}^{s} = 2^{-k} |\rho \cdot \mathds{1}_s| \geq 2^{-k}.
    \end{equation*}
    
    \vspace{0.3cm}

    \noindent\textbf{Case \(k\geq 10\):}
    
    \vspace{0.3cm}

    We take the function defined in~\cite[Display (4.26)]{KP17} with \(\gamma = \frac14\). From their definition, one immediately get Properties 1., 2. and 3. For Property 4., note that \eqref{eq:spin_wave_square_awayfromclust} is shown in \cite[Display above (4.30)]{KP17}, while \eqref{eq:spin_wave_square_atclust} follows from~\cite[First display on p. 28]{KP17} upon observing that $|x_E|\ge 2^{k-1}$.
    Finally, point 5. follows from~\cite[Displays (4.27), (4.28)]{KP17} and the fact that $\frac{\log(6/5)}{4}\approx 0.045$ is larger than $2^{-9}$.
\end{proof}

\begin{proof}[Proof of Lemma~\ref{lem:spin_waves}]
    Let \(\Lambda\subset \Z^2\) be connected. Fix some charge ensemble \(\Ncal\), and \(\rho\in \Ncal\) as given by Theorem~\ref{thm:convex_combination_ensembles}. Let \(b_{\rho,0}\), \(b_{\rho,k}^s\) be the functions given by Lemmas~\ref{lem:spin_wave_level_0}, and~\ref{lem:spin_wave_square} respectively.   
    
    We define \(a_{\Ncal,\rho}\) using
    \begin{equation}\label{eq:def_arho}
        a_{\Ncal,\rho} = \gamma_\beta b_{\rho,0} + \sum_{k= 1}^{\scale_{\Lambda}(\rho)} \sum_{s\in \SquareCover_k^{\sep}(\rho) } \gamma_\beta b_{\rho,k}^s.
    \end{equation}
    Property~\ref{spinWave:support} and~\ref{spinWave:orthogonality} follow from the definition and Lemmas~\ref{lem:spin_wave_level_0}, and~\ref{lem:spin_wave_square}. Property~\ref{spinWave:gradient} follows by noting that for any edge there is at most one function in \eqref{eq:def_arho} whose gradient is non-zero there (see~\cite[After Display (4.35)]{KP17}). Then, the wanted claim follow from Lemmas~\ref{lem:spin_wave_level_0}, and~\ref{lem:spin_wave_square}. 
    
    To establish Property~\ref{spinWave:energy} we use again that in \eqref{eq:def_arho} for any edge there is at most one function in \eqref{eq:def_arho} whose gradient is non-zero. Setting \(a_{\rho}\equiv a_{\Ncal,\rho}\), we obtain
    \begin{equation}\label{eq:decompEbeta}
    \begin{aligned}
        E_{\beta}(a_{\rho},\rho)
        &=
        \gamma_\beta\rho\cdot b_{\rho,0}+\sum_{k= 1}^{\scale_{\Lambda}(\rho)} \sum_{s\in \SquareCover_k^{\sep}(\rho) } \gamma_\beta \rho\cdot b_{\rho,k}^s-(2\Cr{lnIratio_deriv}+1)c_{\beta}\sum_{i\sim j}g(a_{\rho}(i)-a_{\rho}(j))
        \\
        &=
        \gamma_\beta\rho\cdot b_{\rho,0}+\sum_{k= 1}^{\scale_{\Lambda}(\rho)} \sum_{s\in \SquareCover_k^{\sep}(\rho) } \gamma_\beta \rho\cdot b_{\rho,k}^s
        \\
        &\qquad-(2\Cr{lnIratio_deriv}+1)c_{\beta}\sum_{i\sim j}g(\gamma_\beta b_{\rho,0}(i)-\gamma_\beta b_{\rho,0}(j))
        \\
        &\qquad-(2\Cr{lnIratio_deriv}+1)c_{\beta}\sum_{k= 1}^{\scale_{\Lambda}(\rho)} \sum_{s\in \SquareCover_k^{\sep}(\rho) }\sum_{i\sim j}g(\gamma_\beta b_{\rho,k}^s(i)-\gamma_\beta b_{\rho,k}^s(j)).
    \end{aligned}
    \end{equation}
    Lemmas~\ref{lem:spin_wave_level_0} and~\ref{lem:spin_wave_square} together with Lemma \ref{lem:control_A_with_Ssep} imply directly the bound
    \begin{equation}\label{eq:decompEbeta_1stterm}
        \gamma_\beta\rho\cdot b_{\rho,0}+\sum_{k= 1}^{\scale_{\Lambda}(\rho)} \sum_{s\in \SquareCover_k^{\sep}(\rho) } \gamma_\beta \rho\cdot b_{\rho,k}^s\ge \gamma_\beta\Bigg(\frac{\|\rho\|_1}{2}+\sum_{k= 1}^{\scale_{\Lambda}(\rho)} \sum_{s\in \SquareCover_k^{\sep}(\rho) }\frac{1}{2^9}\Bigg) \ge  \gamma_\beta\Bigg (\frac{\|\rho\|_1}{2}   +\frac{A_\Lambda(\rho)}{2^9\Cr{A_Lambda}}\Bigg)
    \end{equation}
    which controls the first term on the right-hand side of \eqref{eq:decompEbeta}. 

    For the second term, we use Properties 1. and 2. of Lemma~\ref{lem:spin_wave_level_0} to estimate
    \begin{equation}\label{eq:decompEbeta_level0}
    \sum_{i\sim j}g(\gamma_\beta b_{\rho,0}(i)-\gamma_\beta b_{\rho,0}(j))\le 4|\supp \rho|g(\gamma_\beta)\le4\|\rho\|_1 g(\gamma_\beta)
    \end{equation}
    It remains to estimate the third term. For that purpose fix some $k\ge1$ and $s\in \SquareCover_k^{\sep}(\rho)$. Let $\Ecal$ be set of connected components of $\bigcup_{\rho'\in \Ncal_{\lesssim}(\rho)}\DensSquareEnv^+(\rho')$. Using Property 3. of Lemma~\ref{lem:spin_wave_square} we see that
    \begin{equation}\label{eq:decompEbeta_square}
    \begin{aligned}
    &\sum_{i\sim j}g(\gamma_\beta b_{\rho,k}^s(i)-\gamma_\beta b_{\rho,k}^s(j))\\
    &\quad\le \sum_{\substack{i\sim j\\i,j\notin\bigcup_{E\in\Ecal}E}}g(\gamma_\beta b_{\rho,k}^s(i)-\gamma_\beta b_{\rho,k}^s(j))+\sum_{E\in\Ecal}\sum_{\substack{i\sim j\\i\in E,j\notin E}}g(\gamma_\beta b_{\rho,k}^s(i)-\gamma_\beta b_{\rho,k}^s(j))\\
    &\quad\le 2\cdot 2^{2(k+1)}g(\gamma_\beta 2^{-k})+\sum_{\substack{E\in\Ecal\\\dist(E,s)\le 2^{k-1}}}|\partial^{\text{ext}}E|g(\gamma_\beta d(E)2^{-k})\\
    &\quad\le 2\cdot 2^{2(k+1)}\cdot\frac{1}{2^{2k}}\Cr{g_quadr}g(\gamma_\beta)+\sum_{\substack{E\in\Ecal\\\dist(E,s)\le 2^{k-1}}}64\frac{d(E)^3}{2^{2k}}\Cr{g_quadr}g(\gamma_\beta)
    \end{aligned}
    \end{equation}
    where we used \eqref{eq:g_quadr} and that $|\partial^{\text{ext}}E|\le 64d(E)$ (as follows from \cite[Lemma 4.6, 4.]{KP17}). 

    By the same argument as for \cite[Display above (4.33)]{KP17}), any two $E,E'\in \Ecal$ with $2^\ell\le d(E),d(E')\le 2^{\ell+1}$ satisfy $\dist(E,E')\ge c2^{\alpha \ell}$, and therefore
    \[\sum_{\substack{E\in\Ecal\\\dist(E,s)\le 2^{k-1}}}d(E)^3\le C\sum_{\ell=0}^\infty \frac{2^{2(k+1)}}{2^{2\alpha \ell}}2^{3(\ell+1)}\le C2^{2k}\sum_{\ell=0}^\infty2^{2-3\alpha)\ell}\le C2^{2k}\]
    Inserting this into \eqref{eq:decompEbeta_square} we obtain that 
    \begin{equation}\label{eq:decompEbeta_square2}
    \sum_{i\sim j}g(\gamma_\beta b_{\rho,k}^s(i)-\gamma_\beta b_{\rho,k}^s(j))\le \Cr{g_quadr}\Cl{spin}g(\gamma_\beta)
    \end{equation}
    for some universal constant $\Cr{spin}$.

    We can now insert the estimates \eqref{eq:decompEbeta_1stterm}, \eqref{eq:decompEbeta_level0} and \eqref{eq:decompEbeta_square2} into \eqref{eq:decompEbeta} to obtain that
    \begin{equation}\label{eq:decompEbeta2}
    \begin{aligned}
        E_{\beta}(a_{\rho},\rho)&\ge \gamma_\beta\Bigg (\frac{\|\rho\|_1}{2}   +\frac{A_\Lambda(\rho)}{2^9\Cr{A_Lambda}}\Bigg)-(2\Cr{lnIratio_deriv}+1)c_{\beta}\Big(4\|\rho\|_1 g(\gamma_\beta)-\Cr{g_quadr}\Cr{spin}g(\gamma_\beta)\Big)\\
        &=\|\rho\|_1\Big(\frac{\gamma_\beta}{2}-4(2\Cr{lnIratio_deriv}+1)c_{\beta}g(\gamma_\beta)\Big)+A_\Lambda(\rho)\Big(\frac{\gamma_\beta}{2^9\Cr{A_Lambda}}-(2\Cr{lnIratio_deriv}+1)c_{\beta}\Cr{g_quadr}\Cr{spin}g(\gamma_\beta)\Big)
    \end{aligned}
    \end{equation}
    Now define the constant $
    \Cr{gamma}=\left(\max\Big(16(2\Cr{lnIratio_deriv}+1),2^{10}\Cr{spin}\Cr{A_Lambda}(2\Cr{lnIratio_deriv}+1)\Big)\right)^{-1}$.
    Because $\Cr{g_quadr}\ge1$, if $\frac{c_{\beta}g(\gamma_\beta)}{\gamma_\beta}\le \frac{\Cr{gamma}}{\Cr{g_quadr}}$, the right-hand side of \eqref{eq:decompEbeta2} is bounded below by
    $\frac{\gamma_\beta\|\rho\|_1}{4}+\frac{\gamma_\beta A_\Lambda(\rho)}{2^{10}\Cr{A_Lambda}}$, and so Property~\ref{spinWave:energy} follows with 
    $\Cr{energy}=\min\Big(\frac14,\frac{1}{2^{10}\Cr{A_Lambda}}\Big)$.

The argument for Property~\ref{spinWave:scalar_prod} is very similar, and actually easier, as instead of $g(\nabla a_{\Ncal,\rho})$ we only need to control $\|\nabla a_{\Ncal,\rho}\|_2^2$. Alternatively, one can use estimates from \cite{KP17}. In fact, the combination of \cite[(4.30) and (4.34)]{KP17} directly yields
    \[
        \norm{\nabla a_{\Ncal,\rho}}_2^2\leq
        C\gamma_{\beta}^2 \Big(|\supp \rho| + \sum_{k= 1}^{\scale_{\Lambda}(\rho)} |\SquareCover_k^{\sep}(\rho)| \Big)
        \leq
        C \gamma_{\beta}^2 \Big(|\supp \rho|+A_{\Lambda}(\rho)\Big)
    \] for some constant \(C\) (that becomes $\Cr{norm_gradSpinWave}$).
\end{proof}

\section{Complex translations}

We prove a simple result of multivariate complex analysis which is used to performed suitable change of integration contours in \(\bbC^n\).

For \(F:D\to \bbC\), \(D\subset \bbC\), \(z\in D\), \(i\in \{1,\dots,n\}\), define \(F_{i}^z\) by
\begin{equation*}
    F_{i}^z(w) = F(z_1,\dots ,z_{i-1},w,z_{i+1}, \dots, z_n).
\end{equation*}
\begin{lemma}
    \label{lem:complex_translations}
    Let \(n\geq 1\). For \(1\leq i,j\leq n\) with \(i\neq j\), let \(\epsilon_{ij}= \epsilon_{ji}\in(0,+\infty]\).
    Define
    \begin{equation*}
        D_{\epsilon} = \{z\in \bbC^n:\ |\Im(z_i-z_j)|< \epsilon_{ij} \ \forall i\neq j\}
        .
    \end{equation*}
    Let \(F:D_{\epsilon}\to \bbC\) be an analytic function. Suppose it is real and integrable on \(\R^n\), with \(|F_{i}^z(x+iy)|\xrightarrow{|x|\to\infty} 0\) uniformly over \(y\) in any fixed compact set, and let \(a\in \R^n\) be such that \(|a_i-a_j|\leq \epsilon_{ij}\) for all \(i\neq j\). Then,
    \begin{equation*}
        \int_{\R^n}F(x)dx = \int_{\R^n}F(x+\rmi a)dx.
    \end{equation*}
\end{lemma}
\begin{proof}
    We will prove the result by successive applications of single coordinate shifts. We denote \(z_i= x_i+\rmi y_i\). Let \(i\in \{1,\dots, n\}\), \(z\in D_{\epsilon}\) and \(\delta>0\) such that \((z_1,\dots,z_{i-1}, z_i + it, z_{i+1},\dots z_n)\in D_{\epsilon}\) for all \(t\in(-\delta,\delta)\). One then has that for any \(t\in (-\delta,\delta)\),
    \begin{equation*}
        \int_{-\infty}^{\infty}dx_i F_i^z(x_i+\rmi y_i) = \int_{-\infty}^{\infty}dx_i F_i^z\big(x_i+\rmi (y_i+t)\big).
    \end{equation*}Indeed, by our hypotheses, \(w\mapsto F_i^z(w+\rmi y_i)\) is analytic on \(\{|\Im(w)|< \delta\}\). So,
    \begin{multline*}
        \int_{-R}^{R} dw F_i^z(w+\rmi y_i + \rmi t) 
        \\
        =
        \int_{-R}^R dw F_i^z(w+\rmi y_i) + \int_{0}^t ds F_i^z(R+\rmi y_i+\rmi s) + \int_{t}^0 ds F_i^z(-R+\rmi y_i+\rmi s).
    \end{multline*}Letting \(R\to\infty\) and using the decay of \(F_i^z(w)\) for \(|\Re(w)|\) large, one gets the wanted identity. We now need to show that translation by \(\rmi a\) can be realized via a sequence of translations of single coordinates as we just described.
    Let
    \begin{equation*}
        \epsilon_* =
        \min\{\epsilon_{ij}:\ 1\leq i< j\leq n \}
        >0.
    \end{equation*}Define \(\delta_i: \{1,\dots, n\} \to \R\) by \(\delta_i(i) = 1\) and \(\delta_i(j) = 0\) if \(j\neq i\). Denote by \(\sfone\) the constant \(1\) function.
    First note that if \(a \equiv \alpha \sfone\) is constant, one can write
    \begin{equation*}
        a = \sum_{k=1}^{\lceil \frac{\alpha}{\epsilon_*}\rceil+1} \alpha_k \sfone,
    \end{equation*}with \(|\alpha_k| < \epsilon_*\). In particular, one can shift by \(a\) by doing \(n(\lceil \frac{\alpha}{\epsilon}\rceil+1)\) single coordinate shifts without leaving \(D_{\epsilon}\). Let us finally show that one can obtain general \(a\)'s as a constant function plus a sequence of single-coordinate shifts without leaving \(D_{\epsilon}\). We construct the shifts in reverse order: we go from \(a\) to a constant function. At each step, shift the set of coordinates with maximal value to the second maximal value by doing single coordinate shifts as in the constant case. This process does not leave \(D_{\epsilon}\) as absolute gradients between coordinates that were not at the maximal value and coordinates that were is only decreasing, and the gradients between coordinates that were at the maximal value are controlled as in the constant case. One such step increase by at least one the number of coordinates having maximal value, and this process stops when the function becomes constant equal to its minimum.
\end{proof}

\section{Analytic extension}
\label{app:analytic_ext}
\begin{proof}[Proof of Proposition \ref{p:heightfunctextension} for the $p$-SOS-model]
We need to find a holomorphic function $I_\beta(z)$ such that $I_\beta(n)=\exp(-\beta|n|^p)$ for $n\in\Z$ and which has all the properties listed in Section \ref{subsec:models}. We assume throughout that $\beta\le1$, say.

We first find a nice holomorphic function $f(z)$ close to $\beta|z|^p$. We let 
\[f(z)=\beta^\alpha\left(1+\beta^{2(1-\alpha)/p}z^2\right)^{p/2}\]
where $0<\alpha<1$ will be fixed later.
We note that the function $f$ is holomorphic on
$\left\{z\in \bbC:\ |\Im(z)| <\frac{1}{\beta^{(1-\alpha)/p}}\right\}$
and that
\begin{align*}
f'(z)&=\frac{p\beta^{\alpha+2(1-\alpha)/p}z}{(1+\beta^{2(1-\alpha)/p}z^2)^{1-p/2}}\\
f''(z)&=\frac{p\beta^{\alpha+2(1-\alpha)/p}+(p-1)p\beta^{\alpha+4(1-\alpha)/p}z^2}{(1+\beta^{2(1-\alpha)/p}z^2)^{2-p/2}}
\end{align*}
If $|\Im(z)| <\frac{1}{2\beta^{(1-\alpha)/p}}$, the denominators are bounded away from 0 by a constant (uniformly in $p$). Moreover the denominator in the expression for $f''$ grows at least at fast as the numerator as $|\Re(z)|\to\infty$ (here we use that $p\le2$), and so 
\begin{equation}\label{e:estf''}
|f''(z)|\le C\beta^{\alpha+2(1-\alpha)/p}
\end{equation}
for $|\Im(z)| <\frac{1}{2\beta^{(1-\alpha)/p}}$.

Next, define $s_n=\beta|n|^p-f(n)$ for $n\in\Z$. Using that $(1+x)^{p/2}=1+O(x)$ for $x\in\R$ with $x\ge0$ (as follows from Taylor's theorem and the fact that the derivative of $(1+x)^{p/2}$ is uniformly bounded for $x\ge0$ and $p\le2$), we have for $n\in\Z\setminus\{0\}$ that
\[|s_n|=\left|\beta|n|^p-\beta|n|^p\left(1+\frac{1}{\beta^{2(1-\alpha)/p}n^2}\right)^{p/2}\right|\le C\beta|n|^p\frac{1}{\beta^{2(1-\alpha)/p}n^2}\le C\frac{\beta^{1-2(1-\alpha)/p}}{|n|^{2-p}}\]
We also have $s_0=-\beta^{\alpha}$, and hence we can combine the bounds into
\[|s_n|\le C\frac{\beta^{\alpha\wedge(1-2(1-\alpha)/p)}}{1+|n|^{2-p}}\quad\forall n\in\Z\]
for $\beta\le1$.

We now define
\begin{equation}\label{e:defI_betaSOS}
I_\beta(z)=\exp\left(-f(z)+\sum_{m\in\Z}s_m\sinc^2(\pi(z-m))\right)
\end{equation}
for $|\Im(z)| <\frac{1}{2\beta^{(1-\alpha)/p}}$, where $\sinc(z)=\frac{\sin(z)}{z}$ (and $\sinc(0)=1$) and claim that this function has all required properties. 

First of all, we note the bounds
\begin{align}
|\sinc(z)^2|&\le C\frac{e^{2\Im(z)}}{1+\Re(z)^2}\label{e:sinc}\\
\left|(\sinc^2)''(z)\right|&\le C\frac{e^{2\Im(z)}}{1+\Re(z)^2}\label{e:sinc''}
\end{align}
The estimate \eqref{e:sinc} ensures that the series in \eqref{e:defI_betaSOS} converges locally uniformly in the strip $\left\{z\in \bbC:\ |\Im(z)| <\frac{1}{2\beta^{(1-\alpha)/p}}\right\}$. This implies that $I_\beta$ is holomorphic in that strip. It is clear that $I_{\beta}$ is even, and real positive on the real line. Moreover, since $\sinc(\pi(n-m))=\delta_{mn}$ for $m,n\in\Z$, one has that $I_\beta(n)=\exp(-\beta|n|^p)$ for $n\in\Z$.

The Assumptions~\ref{item:Ibeta:imaginary_growth}, \ref{item:Ibeta:derivatives} and \ref{item:Ibeta:sub_quad} in Subsection~\ref{subsec:models} all follow from Taylor's theorem and \eqref{e:estf''} and \eqref{e:sinc''}. In detail, we have that
\[\log I_\beta(x+\rmi a)=\log I_\beta(x)+\rmi a(\log I_\beta)'(x)-\frac{a^2}{2}(\log I_\beta)''(\tilde a)\]
for some $\tilde a\in[0,a]$. As $(\log I_\beta)'(x)$ is real, \eqref{e:estf''} and \eqref{e:sinc''} imply that
\begin{align*}
   &\left|\Re\left(\log I_\beta(x+\rmi a)-\log I_\beta(x)\right)\right|\\
   &\le Ca^2\sup_{\tilde a\in[0,a]}\Big(f''(x+\rmi\tilde a)+\sum_{m\in\Z}|s_m|\pi^2|(\sinc^2)''(\pi(x+\rmi\tilde a-m))|\Big)\\
   &\le Ca^2\Big(\beta^{\alpha+2(1-\alpha)/p}+\beta^{\alpha\wedge(1-2(1-\alpha)/p)}e^{2\pi|a|}\sum_{m\in\Z}\frac{1}{1+(x-m)^2}\Big)\\
   &\le Ca^2\big(\beta^{\alpha+2(1-\alpha)/p}+\beta^{\alpha\wedge(1-2(1-\alpha)/p)}e^{2\pi|a|}\big)
\end{align*}
If we choose $\alpha=1-\frac{p}{3}$, then $\alpha+2(1-\alpha)/p\ge\alpha\ge\frac13$ and $\alpha\wedge(1-2(1-\alpha)/p)=\frac13$, and hence Assumption~\ref{item:Ibeta:imaginary_growth} follows with $g(a)=a^2(1+e^{2\pi|a|})$, $c_\beta=C\beta^{1/3}$, and $\epsilon_\beta=\frac{1}{2\beta^{1/3}}$ (recall that $\beta\le1$).

Similarly, we find that
\begin{align*}
   &\left|\frac{\dd}{\dd x}\log I_\beta(x+\rmi a)-\frac{\dd}{\dd x}\log I_\beta(x)\right|\\
   &\le C|a|\sup_{\tilde a\in[0,a]}\Big(|f''(x+\rmi\tilde a)|+\sum_{m\in\Z}|s_m|\pi^2|(\sinc^2)''(\pi(x+\rmi\tilde a-m))|\Big)\\
   &\le C|a|\big(\beta^{\alpha+2(1-\alpha)/p}+\beta^{\alpha\wedge(1-2(1-\alpha)/p)}e^{2\pi|a|}\big)
\end{align*}
and
\begin{align*}
   &\left|\frac{\dd^2}{\dd x^2}\log I_\beta(x+\rmi a)-\frac{\dd^2}{\dd x^2}\log I_\beta(x)\right|\\
   &\le C\sup_{\tilde a\in[0,a]}\Big(|f''(x+\rmi\tilde a)|+\sum_{m\in\Z}|s_m|\pi^2|(\sinc^2)''(\pi(x+\rmi\tilde a-m))|\Big)\\
   &\le C\left(\beta^{\alpha+2(1-\alpha)/p}+\beta^{\alpha\wedge(1-2(1-\alpha)/p)}e^{2\pi|a|}\right)
\end{align*}
which imply Assumption~\ref{item:Ibeta:derivatives} with $c_\beta'=C\beta^{1/3}$.

Finally, we estimate that
\begin{align*}
   &\left|\frac{\dd^2}{\dd x^2}\log I_\beta(x)\right|\\
   &\le |f''(x)|+\sum_{m\in\Z}|s_m|\pi^2|(\sinc^2)''(\pi(x-m))|\\
   &\le C\left(\beta^{\alpha+2(1-\alpha)/p}+\beta^{\alpha\wedge(1-2(1-\alpha)/p)}\right)
\end{align*}
which implies Assumption~\ref{item:Ibeta:sub_quad} with $c_\beta'=C\beta^{1/3}$.

\end{proof}

\subsubsection*{Acknowledgements}
FS was supported by the NCCR SwissMAP, the Swiss FNS, and the Simons collaboration on localization of waves.

\end{document}